\documentclass{article}[12pt]
\usepackage{latexsym}
\usepackage{amsmath}
\usepackage{amsthm}
\usepackage{amsfonts}
\usepackage{amssymb}
\usepackage{graphicx}
\usepackage[top=2.5cm, bottom=2.5cm, left=2.5cm, right=2.5cm]{geometry}
\usepackage{hyperref}
\usepackage{verbatim}
\usepackage{slashed,mathtools,color}
\usepackage[titletoc,toc]{appendix}
\usepackage{float}
\newcommand{\pfstep}[1]{\vspace{.5em} {\it \noindent #1.}}

\newcommand{\bD}{{\bf D}}
\newcommand{\bg}{{\bf g}}

\newcommand{\beq}{\begin{equation}}
\newcommand{\eeq}{\end{equation}}



\theoremstyle{plain}
\newtheorem{theorem}{Theorem}[section]
\newtheorem{proposition}[theorem]{Proposition}
\newtheorem{lemma}[theorem]{Lemma}

\theoremstyle{definition}
\newtheorem{definition}[theorem]{Definition}

\theoremstyle{remark}
\newtheorem{remark}[theorem]{Remark}

\numberwithin{equation}{section}

\allowdisplaybreaks[4] 

\swapnumbers
\begin{document}

\title{The initial boundary value problem for the Einstein equations with totally geodesic timelike boundary}

\author{Grigorios Fournodavlos,\footnote{Sorbonne Universit\'e, CNRS, Universit\'e  de Paris, Laboratoire Jacques-Louis Lions (LJLL), F-75005 Paris, France,
		grigorios.fournodavlos@sorbonne-universite.fr}\;\; Jacques Smulevici\footnote{Sorbonne Universit\'e, CNRS, Universit\'e  de Paris, Laboratoire Jacques-Louis Lions (LJLL), F-75005 Paris, France, jacques.smulevici@sorbonne-universite.fr}}

\date{}

\maketitle

\begin{abstract}
We prove the well-posedness of the initial boundary value problem for the Einstein equations with sole boundary condition the requirement that the timelike boundary is totally geodesic. This provides the first well-posedness result for this specific geometric boundary condition and the first setting for which geometric uniqueness in the original sense of Friedrich holds for the initial boundary value problem. 

Our proof relies on the ADM system for the Einstein vacuum equations, formulated with respect to a parallelly propagated orthonormal frame along timelike geodesics. As an independent result, we first establish the well-posedness  in this gauge of the Cauchy problem for the Einstein equations, including the propagation of constraints. 
More precisely, we show that by appropriately modifying the evolution equations, using the constraint equations, we can derive a first order symmetric hyperbolic system for the connection coefficients of the orthonormal frame. The propagation of the constraints then relies on the derivation of a hyperbolic system involving the connection, suitably modified Riemann and Ricci curvature tensors and the torsion of the connection. In particular, the connection is shown to agree with the Levi-Civita connection at the same time as the validity of the constraints.


In the case of the initial boundary value problem with totally geodesic boundary, we then verify that the vanishing of the second fundamental form {}{of} the boundary leads to homogeneous boundary conditions for our modified ADM system, as well as for the hyperbolic system used in the propagation of the constraints. An additional analytical difficulty arises from a loss of control on the {}{normal} derivatives to the boundary of the {}{solution}. 
{}{To resolve this issue, we work with} an anisotropic scale of Sobolev spaces and {}{exploit} the specific structure of the equations. 
\end{abstract}

\tableofcontents

\parskip = 2 pt


%
\section{Introduction}

This article establishes the well-posedness of the initial boundary value problem (IBVP) for the Einstein {}{vacuum} equations
\begin{align}\label{EVE}
{\bf \mathrm{Ric}(g)}=0,
\end{align}
 in the specific case of a \emph{totally geodesic} timelike boundary.

\subsection{The initial boundary value problem in General Relativity}

In the standard formulation of the Cauchy problem for the Einstein {}{vacuum} equations, given a Riemannian manifold $(\Sigma, h)$ and a $2$-tensor $K$ satisfying the constraints equations 
\begin{align}
\label{Hamconst}R-|{k}|^2+(\text{tr}{k})^2=&\,0,\\
\label{momconst}\mathrm{div}{k}-\mathrm{d}\text{tr}{k}=&\,0,
\end{align}
where $R$ is the scalar curvature of the Riemannian metric $h$ and all operators are taken with respect to $h$, the goal is to construct a Lorentzian manifold $(\mathcal{M},{\bf g})$ solution to the Einstein equations, together with an embedding of $\Sigma$ into $\mathcal{M}$ such that $(h,{k})$ coincides with the first and second fundamental form of the embedding. For the IBVP, we now require that $\Sigma$ is a manifold with boundary $\mathcal{S}$ and we must also complete the initial data ${}{(h,{k})}$ with a set of boundary conditions. A solution to the IBVP is then a Lorentzian manifold $(\mathcal{M},{\bf g})$ with a timelike boundary  $\mathcal{T}$ diffeormorphic to $\mathcal{S} \times{}{\mathbb{R}}$ such that as before, there exists some embedding of $\Sigma$ into $\mathcal{M}$ respecting the {}{initial} data, with the image of $\partial \Sigma$ lying on $\mathcal{T}$, and such that the boundary conditions are verified on $\mathcal{T}$. On top of the constraint equations, the initial data must also now verify the so-called corner or compatibility conditions on $\mathcal{S}$. 


The IBVP is {}{related} to many important aspects of {}{general relativity} and the Einstein equations such as numerical relativity, the construction of asymptotically Anti-de-Sitter spacetimes, timelike hypersurfaces emerging as the boundaries of the support of massive matter fields or the study of gravitational waves in a cavity and their nonlinear interactions. This problem was first addressed for the Einstein equations in the seminal work of Friedrich-Nagy \cite{FriedNag}, as well as by Friedrich \cite{Fried95} in the related Anti-de-Sitter setting.\footnote{See also \cite{CarVal, EncKam} for extensions and other proofs of well-posedness in the Anti-de-Sitter case.} Well-posedness {}{of} the IBVP has since been obtained in generalized wave coordinates, cf~\cite{KRSW} or the recent \cite{AA},\footnote{To be more precise, the boundary data in \cite{AA} relies on an auxiliary wave map equation akin to generalized wave coordinates. This introduces a geometric framework to address the IBVP, albeit for the Einstein equations coupled to the auxiliary wave map equation. } and for various first and second order systems derived from the ADM formulation of the Einstein equations, see for instance \cite{FSm, SarTig} {}{and previous work in numerics \cite{AndYork,FriReu}}. We refer to \cite{SarTig2} for an extensive review of the subject. 


\subsection{Geometric uniqueness} 

One of the {}{remaining} outstanding issues, concerning the study of the Einstein equations in the presence of a timelike boundary, is the geometric uniqueness problem of Friedrich {}{\cite{Fried09}}. Apart from the construction of asymptotically Anti-de-Sitter spacetimes {}{\cite{Fried95}}, where the timelike boundary is a conformal boundary at spacelike infinity, all results establishing well-posedness, for some formulations of the {}{IBVP}, impose certain gauge conditions on the boundary, and the boundary data depend on these choices. In particular, given a solution to the Einstein equations with a timelike boundary, different gauge choices will lead to different boundary data, in each of the formulations for which well-posedness is known. On the other hand, if we had been given the different boundary data a priori, we would not know that these lead to the same solution. The situation is thus different from the usual initial value problem, for which only isometric data {}{lead} to isometric solutions, which one then regards as the same solution.

In the Anti-de-Sitter setting, this problem admits one solution: in [9], Friedrich proved that one can take the conformal metric of the boundary as boundary data, which is a geometric condition independent of any {}{gauge}.\footnote{Note that, even in the Anti-de-Sitter setting, it is possible to formulate other boundary conditions, such as dissipative boundary conditions, for which one knows how to prove some sort of well-posedness, however, with a formulation of the boundary conditions that is gauge dependent and thus, such that we do not know whether geometric uniqueness holds or not.}

\subsection{The IBVP with totally geodesic boundary}

Our main result concerning the IBVP can be formulated as follows. 
\begin{theorem}\label{thmA}
Let $(\Sigma, h, {k})$ be a smooth initial data set for the Einstein {}{vacuum} equations such that $\Sigma$ is a $3$-manifold with boundary ${}{\partial\Sigma=\mathcal{S}}$. Assume that the corner conditions of Lemma \ref{lem:compcond.angle} hold on ${}{\mathcal{S}}$, with respect to a given angle $\omega$, where $\omega$ is a smooth function defined on  $\mathcal{S}$. Then, there exists a smooth Lorentzian manifold $({}{\mathcal{M}}, \bf g)$ solution to the Einstein {}{vacuum} equations with boundary $\partial \mathcal{M}= \widehat{\Sigma} \cup \mathcal{T}$ such that 
\begin{enumerate}
\item there exists an embedding $i$ of $\Sigma$ onto $\widehat{\Sigma}$ with $(h,{k})$ coinciding with the first and second fundamental form of the embedding, \label{th:d}
\item $\mathcal{T} \cap \widehat{\Sigma}=i( \mathcal{S})$ and $\mathcal{T}$ is a timelike hypersurface emanating from $i(\mathcal{S})$ at an angle $\omega$ relative to $\hat{\Sigma}$, \label{th:tb}
\item $\mathcal{T}$ is totally geodesic, i.e.~it has vanishing second fundamental form {}{$\chi$}, \label{th:tg}
\item geometric uniqueness holds: given any other solution $(\mathcal{M}', \bf g')$ verifying \ref{th:d}, \ref{th:tb} and \ref{th:tg}, $(\mathcal{M}, \bf g)$ and $(\mathcal{M}', \bf g')$ are both extensions\footnote{Recall that $(\mathcal{M}, g)$ is an extension of $(\mathcal{M}'', g'')$, if there exists an isometric {}{embedding} ${}{\psi}: \mathcal{M}'' \rightarrow \mathcal{M}$, preserving orientation, and such that $\psi \circ i''=i$, where $i'': \Sigma \rightarrow \mathcal{M}''$ is the embedding of the initial hypersurface into $\mathcal{M}''$.}   of yet another solution verifying \ref{th:d}, \ref{th:tb} and \ref{th:tg}. 
\end{enumerate}
\end{theorem}
\begin{remark} \label{rk:pibvp}
The above theorem is obtained using a system of reduced equations based on the ADM system in a geodesic gauge. For the reduced equations, due to the presence of a boundary and our choice of boundary conditions, we prove local well-posedness in a scale of anisotropic Sobolev spaces, see Definition \ref{def:Bs} and {}{Proposition \ref{prop:bdlocex}}. Indeed, the boundary conditions can a priori only be commuted by tangential derivatives to the boundary. Thus, our Sobolev spaces distinguish between derivatives tangential and {}{normal} to the boundary. In view of this, the {}{normal} derivatives cannot be estimated using commutation and standard energy estimates, but instead, are recovered from the equations directly, which allow to rewrite {}{normal} derivatives in terms of tangential ones. However, the structure of the equations plays an essential role here, since some components do not have any {}{normal} derivatives appearing in the equations. The anisotropic Sobolev spaces provide a solution to this analytical problem, cf.~proof of {}{Proposition \ref{prop:bdlocex}.} 
\end{remark}
\begin{remark}
Since the reduced system is solved in (anisotropic) Sobolev spaces, one can obtain a similar statement assuming only that the initial data lie in a standard $H^s$ space, $s \ge 7$, with corner conditions satisfied up to the corresponding finite order. 
\end{remark}

\begin{remark}
The angle $\omega$ measures the (hyperbolic) rotation that makes the unit normals to $\Sigma$ on $\mathcal{S}$ and to $\mathcal{S}$ within $\Sigma$ adapted to the boundary, i.e. tangential and normal to $\mathcal{T}$ respectively, see \eqref{omega} and Figure \ref{fig-Sigma}. 
Although we first write the corner conditions for $\Sigma$ with an angle in Lemma \ref{lem:compcond.angle}, we then switch to an orthogonal slice $\Sigma_0$ to $\mathcal{T}$ and reformulate the induced corner conditions for $\Sigma_0$ in Lemma \ref{lem:compcond}.
\end{remark}

\begin{remark}
Note that, importantly, our choice of boundary conditions {}{for the Einstein equations} translates to admissible boundary conditions both for the reduced system of evolution equations that we use to construct a solution {}{(see Lemma \ref{lem:bdcond})} and for the hyperbolic system that allows a posteriori to prove the propagation of constraints {}{(see Lemma \ref{lem:bdcond.Ric})} and recover the Einstein equations. More precisely, $\chi\equiv0$ on the boundary implies the validity of the momentum constraint,\footnote{Note that we are not referring to \eqref{momconst} here, but to the analogous constraint equations where $k$ is replaced by $\chi$ and $\mathrm{div}$ and $\mathrm{tr}$ are the divergence and trace with respect to the induced metric on the boundary.} 
which translates to homogeneous boundary conditions for certain Ricci components.  
\end{remark}

\begin{remark}
The geometric uniqueness is a direct consequence of our choice of geometric boundary conditions. Although totally geodesic boundaries are of course quite special, this result provides the first setting in which geometric uniqueness holds for the Einstein {}{vacuum} equations with zero cosmological constant $\Lambda=0$.
\end{remark}

\begin{remark}
Note that any Lorentzian manifold admitting a spacelike Killing vector field {}{which is also hypersurface orthogonal} provides an example with such a totally geodesic boundary. For instance, any constant $t$ hypersurface in the interior of {}{a Schwarzschild} black hole can be seen as such a hypersurface. 
\end{remark}

\begin{remark}
If one thinks {}{of} $\chi\equiv0$ as the vanishing of {}{the Lie derivative of the solution in the normal direction to the boundary}, our boundary conditions could be interpreted as homogeneous Neumann boundary conditions, and, in this respect, a natural direction for possible extensions of this result would be to consider inhomogeneous Neumann type boundary conditions, for instance by prescribing a non-zero $\chi$. However, there {}{seem} to be {}{nontrivial} obstructions for such {}{type} of results to hold, both {}{analytic}, due to various losses of derivatives, and {}{geometric}, since geodesics of the boundary are no longer geodesics of the Lorentzian manifold. On a more physical point of view, note that if one thinks of $\chi\equiv0$ as a form of homogeneous Neumann boundary conditions, our setting is applicable {}{to} the study of gravitational waves in a cavity.
\end{remark}

\begin{remark}
	Recall that if $\phi: (\mathcal{M},{\bf g}) \rightarrow (\mathcal{M},{\bf g})$ is an isometry of a Riemanian or Lorentzian manifold, then every connected component of the set of fixed points $\{ p \in \mathcal{M}\,: \,\phi(p)=p \}$ is totally geodesic \cite{WK}. This suggests\footnote{We would like to thank M.T.~Anderson and E.~Witten for this suggestion.} another possible proof of Theorem \ref{thmA}, at least in the case where the initial data intersect the boundary orthogonally, based on extending the initial data via reflection, then solving the regular Cauchy problem for the extended data and finally checking that the resulting spacetime enjoys a discrete isometry. Of course, this approach is clearly not generalizable to other kind of boundary conditions, while the proof of this paper may serve as a basis for further applications in the subject. 
	\end{remark}

\subsection{The hyperbolicity of the ADM system in a geodesic gauge}

As already explained, our choice of evolution equations is based on the ADM formulation of the Einstein equations. 
This formalism and its many variants are widely used in the study of the Einstein equations, by numerical or theoretical means. They are based on a $3+1$ splitting of the underlying Lorentzian manifold $(\mathcal{M},{\bf g})$ through a choice of time function $t$ and the foliation induced by its level sets $\Sigma_t$. The main dynamical variables are then the first and second fundamental forms $(g,K)$ of each $\Sigma_t$, that satisfy, together with the lapse and shift of the foliation, a system of partial differential equations, which is first order in the time derivative. This system is generally underdetermined due to the geometric invariance of the equations. In order to render it well-determined, one naturally needs to make additional gauge choices, leading to a reduced system of equations. In full generality, they are many possible such choices, see for example \cite{Fried96} and the references therein. However, the well-posedness problem has only been rigorously studied so far in certain specific cases, as in \cite{AndMon, SarTig}.

In this paper, we consider the reduced ADM system for the Einstein vacuum equations, 
obtained by writing the equations in an orthonormal frame $\{e_\mu\}_{\mu=0}^3$, which is parallelly propagated with respect to a family of timelike geodesics. In this setting, the lapse of the foliation  is fixed to $1$, while the shift is set to zero, and the spacetime metric takes the form\footnote{Here the Einstein summation is used for the Latin indices that range in 1,2,3.} 
\begin{align}\label{metric}
{\bf g}=-dt^2+g_{ij}dx^idx^j,
\end{align}
where $(x^1,x^2,x^3)$ are $t$-transported coordinates, with respect to which the orthonormal frame is expressed via
\begin{align}\label{fij}
e_0=\partial_t,\qquad e_i=f_i{}^j\partial_j,\qquad\partial_j=f^b{}_j\, e_b,\qquad i,j=1,2,3,
\end{align}
where $f_i{}^jf^b{}_j=\delta_i^b$, $f^b{}_jf_b{}^i=\delta_j^i$.

In the classical ADM formalism, the main evolution equations are first order equations (in $\partial_t$) for $g_{ij},\partial_tg_{ij}$; the second variable corresponding to the second fundamental form of $\Sigma_t$. When expressed in terms of the previous orthonormal frame, $g_{ij}$ correspond to the frame coefficients $f_i{}^j$, while the second fundamental form is now evaluated against the spatial frame components $e_i$:
\begin{align}\label{Kij}
K_{ij}:={\bf g}({\bf D}_{e_i}e_0,e_j)=K_{ji},
\end{align}
where ${\bf D}$ is the Levi-Civita connection of ${\bf g}$. In our framework, $K_{ij},f_i{}^j$ satisfy the reduced equations \eqref{e0Kij}, \eqref{e0fij}. The right-hand-side of \eqref{e0Kij} contains up to two spatial derivatives of $f_i{}^j$, encoded in the Ricci tensor of $g$. However, we find it analytically convenient to expand this term using the spatial connection coefficients
of the frame:\footnote{Although we do not use this anywhere, we note that they can be computed purely in terms of $f_i{}^j$ using the Koszul formula:
\begin{align*}
\Gamma_{ijb}=\frac{1}{2}\bigg[g([e_i,e_j],e_b)-g([e_j,e_b],e_i)+g([e_b,e_i],e_j)\bigg]
=\frac{1}{2}\bigg[f^b{}_l(e_if_j{}^l-e_jf_i{}^l)
-f^i{}_l(e_jf_b{}^l-e_bf_j{}^l)
+f^j{}_l(e_bf_i{}^l-e_if_b{}^l)\bigg]
\end{align*}
}
\begin{align}\label{Gammaijb}
\Gamma_{ijb}:={\bf g}({\bf D}_{e_i}e_j,e_b)=g(D_{e_i}e_j,e_b)=-\Gamma_{ibj},
\end{align}
where $D$ is the Levi-Civita connection of $g$. These then satisfy the propagation equation \eqref{e0Gammaijb}.

At first glance, the system \eqref{e0Kij}-\eqref{e0fij} does not seem to be eligible for an energy estimate, due to the first term in the right-hand-side of \eqref{e0Kij} that renders the system non-symmetric and could lead to a loss of derivatives. This is a well-known problem of the ADM system. One remedy is to consider a harmonic gauge \cite{AndMon} on the slices $\Sigma_t$, which would eliminate this bad term. Another argument was given in \cite{RodSp}, where the authors considered a CMC foliation and made use of the momentum constraint \eqref{momconst}, in order to eliminate any such bad terms in the energy estimates by integrating by parts.\footnote{In \cite{RodSp}, the authors expressed the evolution equation \eqref{e0Kij} in terms of a transported coordinate system $(t,x_1,x_2,x_3)$ and their associated Christoffel symbols. Moreover, they proved a priori energy estimates \emph{assuming} the existence of a solution verifying the constraints, instead of deriving a system for which well-posedness holds, as we do in this paper.} {}{The adoption of such gauges introduces new variables to the system (lapse, shift vector field) that satisfy elliptic equations.}

In contrast, the ADM system can be transformed into a second order system of equations for the second fundamental form of the time slices, expressed in terms of transported coordinates $(t,x^1,x^2,x^3)$. This was first derived in \cite{ChoqRug}, where the authors demonstrated its hyperbolicity under the gauge assumption $\square_gt=0$. It turns out that the second order system for $K$ is also hyperbolic in normal transported coordinates \eqref{metric}, without any additional gauge assumptions, see the framework presented in \cite{FLuk} with an application to asymptotically Kasner-like singularities. Recently, we also used the aforementioned second order system for $K$ (see \cite{FSm}) to analyse the initial boundary value problem for the Einstein vacuum equations in the maximal gauge. 

In the present study, we carry out the analysis in the geodesic gauge presented above, circumventing the apparent loss of derivatives issue (see Lemma \ref{lem:loss}) by making use of both the Hamiltonian and momentum constraints.\footnote{In \cite{FriReu, SarTig}, both the Hamiltonian and momentum constraints were already used to modify the ADM system and obtain well-posedness of the equations in {}{coordinate-based} gauges. The orthonormal frame {}{that we consider in the present article} seems to considerably simplify the analysis of the boundary conditions in our setting.}
More precisely, we prove that by modifying \eqref{e0Kij}-\eqref{e0Gammaijb}, adding appropriate multiples of \eqref{e0Kijnew}-\eqref{e0Gammaijbnew}, 
one obtains a first order symmetric hyperbolic system for the unknowns, see \eqref{e0Kijnew}-\eqref{e0Gammaijbnew}, which is suitable for a local existence argument. In order to {}{facilitate} the propagation of the {}{(anti)}symmetries of $K$ and $\Gamma$, we also {}{(anti)}symmetrize part of the equations. 

In general, once the reduced system is solved, one then recovers the Einstein equations through the Bianchi equations. For a modified system, however, the equations one solves for are not directly equivalent to the vanishing of the components of the Ricci tensor and thus this procedure becomes more complicated.\footnote{See {}{\cite[Appendix A]{SarTig}} for such an example concerning the progragation of constraints in {}{a modified} ADM setting.} It is for this reason that one should make minimal modifications to the reduced equations, since any additional change could complicate even further the final system for the vanishing quantities, making it intractable via energy estimates.  Nonetheless, for the modified system we consider, we are able to recover the full Einstein equations by deriving a hyperbolic system for appropriate combinations of the vanishing quantities {}{(see Lemma \ref{lem:systRic})}. Note that since the connection is obtained by solving the {}{modified} reduced equations, it can only be shown to agree with the Levi-Civita connection at the same time as the recovery of the full Einstein equations {}{(see Section \ref{sec:EVEsol})}. This issue was already present in the approach of \cite{FriedNag} using an orthonormal frame. In particular, it is not known a priori that the torsion of the connection vanishes. Thus, the unknowns in the hyperbolic system used for the recovery of the Einstein equations are the components of the torsion, as well as the components of the Ricci and Riemann tensors, after suitable symmetrizations and modifications. The modifications involve the torsion and are {}{similar\footnote{{}{In \cite{FriedNag}, the authors study the Einstein equations at the level of the Bianchi equations, which results into a different system for the recovery of the Einstein equations.}}} to the modifications used in \cite{FriedNag}. 

Our result on the well-posedness of the Einstein equations in the above framework can then be stated as follows
\begin{theorem}\label{thmB}
The modified reduced system \eqref{e0fij}, \eqref{e0Kijnew}, \eqref{e0Gammaijbnew}, for the frame and connection coefficients is locally well-posed in $L^\infty_tH^s(\Sigma_t)$, for $s\ge3$. Moreover, if the initial data $(\Sigma,h,K)$ satisfy the constraint equations \eqref{Hamconst}-\eqref{momconst}, then the solution to \eqref{e0fij}, \eqref{e0Kijnew}, \eqref{e0Gammaijbnew}, with the induced initial data (see Section \ref{subsec:ID}), induces a solution of \eqref{EVE}.
In particular, the Einstein vacuum equations, cast as a modified ADM system, are locally well-posed. 
\end{theorem}
\begin{remark}
Note that the geodesic gauge considered here respects the hyperbolicity of the equations. In particular, the usual finite speed of propagation and domain of dependence arguments can be proven in this gauge. {}{Hence,} in the case of the initial boundary value problem, one can localize the analysis near {}{a point on} the boundary, {}{provided that the orthonormal frame we consider is adapted to the boundary. This requirement is verified for vanishing $\chi$ (Lemma \ref{lem:adframe}).} 
\end{remark}

\subsection{Outline}
In Section \ref{sec:Framework}, we set up our modified version of the ADM system. We first {}{formulate} the standard ADM evolution equations in the geodesic gauge (Lemma \ref{lem:evoleq}) and then prove (in Lemma \ref{lem:loss}) a first order energy identity, assuming that the constraints hold. This identity {}{leads} us to the introduction of the modified  evolution equations \eqref{e0Kijnew}-\eqref{e0Gammaijbnew}. The resulting system is then shown to be symmetric hyperbolic in Lemma \ref{lem:hypsymm}. {}{Although its local well-posedness follows from standard arguments, {}{to simplify the treatment of the IBVP}, we
establish localized energy estimates in Section \ref{se:lwpde} (cf~Proposition \ref{prop:locex}), using the structure of the commuted equations identified in Section \ref{subsec:diffsyst}.} In Section \ref{subsec:ID}, we briefly describe how to derive the initial data for the reduced system from the geometric initial data. 

Section \ref{sec:totgeodbd} is devoted to the initial boundary value problem for the modified ADM system. First, in Section \ref{lem:bcgf}, using in particular that {}{our} geodesic {}{frame is} adapted to the totally geodesic boundary (Lemma \ref{lem:adframe}), we express the {}{vanishing of $\chi$} in terms of the frame components of $K$ and $\Gamma$. We then prove, in Section \ref{se:lwibvp}, the local well-posedness of the initial boundary value problem under our choice of boundary conditions. The main difficulty here arises from a loss {}{in the control of} the {}{normal} derivatives to the boundary, cf~Remark \ref{rk:pibvp}, {}{forcing} us to the introduction of anisoptropic Sobolev spaces. 

Finally, in Section \ref{sec:EVEsol}, we show that {}{once} a solution to the reduced system has been obtained, our framework allows for the recovery of the Einstein vacuum equations, both for the standard Cauchy problem and in the presence of a totally geodesic boundary, thus completing the proofs of Theorems \ref{thmA} and \ref{thmB}. The starting point is to introduce the Lorentzian metric and the connection associated to a solution of the reduced equations. One easily verifies that the connection is compatible with the metric, {}{by virtue of the propagation of the antisymmetry of the spatial connection coefficients $\Gamma_{ijb}$} (see Lemma \ref{lem:symm}). On the other hand, the connection is not a priori torsion free and therefore, does not a priori agree with the Levi-Civita {}{of} the metric. {}{We first} derive various geometrical identities such as the Bianchi equations and the Gauss-Codazzi equations, in the presence of torsion (cf~Lemma \ref{lem:Dtilde}). Since the resulting equations are not suitable to propagate the constraints, we consider modified Riemann and Ricci curvature tensors \eqref{curvhat}, both for the spacetime geometry and the geometry of the time slices, the modifications depending on the torsion {}{(cf. \cite[Section 6]{FriedNag})}. The symmetries of these modified curvatures are studied in Lemma \ref{lem:Dhat} and \ref{lem:cyclRiem}. {}{Then, we prove that they} lead to a symmetric hyperbolic system \eqref{e0Cijb}-\eqref{Rijeq} for the modified spacetime Ricci curvature components and the torsion. Finally, we {}{show that the boundary conditions satisfied by the solution to the modified ADM system, which are in turn induced by the vanishing of $\chi$ (see Lemma \ref{lem:bdcond}), imply boundary conditions for the modified spacetime Ricci curvature (Lemma \ref{lem:bdcond.Ric})} that are suitable for an energy estimate. {}{The final argument for the recovery of the Einstein equations, both for the Cauchy problem and in the case a totally geodesic timelike boundary, is presented in Section \ref{subsec:finstep}. }

\subsection{Notation}

We will in general use Greek letters $\alpha,\beta,\gamma,\mu,\nu$ for indices ranging from $0$ to $3$, Latin letters $i,j,l,a,b,c$ etc, as spatial indices $1,2,3$, and capital letters $A,B$ for the indices $1,2$ (which correspond below to spacelike vector fields tangential to the boundary). Whenever the Einstein summation is used, the range of the sum will be that of the specific indices. All tensors throughout the paper are evaluated against an orthonormal frame $\{e_\mu\}^3_0$. In particular, we raise and lower indices using $m_{ab}=\mathrm{diag}(-1,1,1,1)$. For example, $e^b=e_b$, $e^0=-e_0$.

\subsection{Acknowledgements}
G.F. would like to thank Jonathan Luk for useful discussions. We would also like to thank M.T.~Anderson for several interesting comments on our work. 
Both authors are supported by the \texttt{ERC grant 714408 GEOWAKI}, under the European Union's Horizon 2020 research and innovation program.

%
\section{The ADM system in a geodesic gauge}\label{sec:Framework}

In this section we introduce our framework and show that the Einstein vacuum equations (EVE) reduce to a first order symmetric hyperbolic system for the connection coefficients of a parallelly propagated orthonormal frame. For completeness, we confirm its well-posedness in usual $H^s$ spaces.

 \subsection{The modified ADM evolution equations and their hyperbolicity} \label{se:madh}

Let $(\mathcal{M},{\bf g})$ be a $3+1$-dimensional Lorentzian manifold and let $\Sigma_0$ be a Cauchy hypersurface equipped with an orthonormal frame $e_1,e_2,e_3$. Also, let $e_0$ be the future unit normal to $\Sigma_0$. We extend the frame $\{e_\mu\}_0^3$ by parallel propagation along timelike geodesics emanating from $\Sigma_0$ with initial speed $e_0$: 
\begin{align}\label{geodgauge}
{\bf D}_{e_0}e_\mu=0
\end{align}
If $t$ is the proper time parameter of the $e_0$ geodesics, $\{t=0\}=\Sigma_0$, then ${\bf g}$ takes the form \eqref{metric}, where $g$ is the induced metric on $\Sigma_t$, and the {}{transition between $\{e_\mu\}_0^3$ and a transported coordinate system $(t,x_1,x_2,x_3)$ is defined via \eqref{fij}. The connection coefficients of the orthonormal frame are $K_{ij},\Gamma_{ijb}$, defined in \eqref{Kij}, \eqref{Gammaijb}.}

Our convention for the spacetime Riemann, Ricci, and scalar curvatures is
\begin{align}\label{curvconv}
{\bf R}_{\alpha\beta\mu\nu}={\bf g}(({\bf D}_{e_\alpha}{\bf D}_{e_\beta}-{\bf D}_{e_\beta}{\bf D}_{e_\alpha}-{\bf D}_{[e_\alpha,e_\beta]})e_\mu,e_\nu),&&{\bf R}_{\beta\mu}={\bf R}_{\alpha \beta\mu}{}^\alpha,\qquad {\bf R}={\bf R}_\mu{}^\mu 
\end{align}
and similarly for the curvature tensors of $g$, denoted by $R_{ijlb},R_{jl},R$. 
\begin{lemma}\label{lem:GaussCod}
With the above conventions, the Gauss and Codazzi equations for $\Sigma_t$ read:
\begin{align}
\label{Gauss}
{\bf R}_{aijb}=&\,R_{aijb}+K_{ab}K_{ij}-K_{aj}K_{ib},\\
\label{R0ijb}{\bf R}_{0ijb}=&\,D_jK_{bi}-D_bK_{ji},
\end{align}
where 
\begin{align}\label{Raijb}
R_{aijb}=e_a\Gamma_{ijb}-e_i\Gamma_{ajb}-\Gamma_{ab}{}^c\Gamma_{ijc}+\Gamma_{ib}{}^c\Gamma_{ajc}-\Gamma_{ai}{}^c\Gamma_{cjb}+\Gamma_{ia}{}^c\Gamma_{cjb}
\end{align}
\end{lemma}
\begin{proof}
We employ the formulas $${\bf D}_{e_i}e_j=D_{e_i}e_j+K_{ij}e_0,\qquad {\bf D}_{e_b}e_0=K_b{}^ce_c,\qquad [e_j,e_b]=D_{e_j}e_b-D_{e_b}e_j=(\Gamma_{jb}{}^c-\Gamma_{bj}{}^c)e_c$$ 
to compute
\begin{align*}
{\bf R}_{aijb}=&\,{\bf g}(({\bf D}_{e_a}{\bf D}_{e_i}-{\bf D}_{e_i}{\bf D}_{e_a}-{\bf D}_{[e_a,e_i]})e_j,e_b)\\
=&\,{\bf g}({\bf D}_{e_a}(D_{e_i}e_j+K_{ij}e_0)-{\bf D}_{e_i}(D_{e_a}e_j+K_{aj}e_0)-D_{[e_a,e_i]}e_j,e_b)\\
=&\,g(D_{e_a}D_{e_i}e_j-D_{e_i}D_{e_a}e_j-D_{[e_a,e_i]}e_j,e_b)+K_{ij}K_{ab}-K_{aj}K_{ib},
\end{align*}
\begin{align*}
{\bf R}_{0ijb}={\bf R}_{jb0i}=&\,{\bf g}(({\bf D}_{e_j}{\bf D}_{e_b}-{\bf D}_{e_b}{\bf D}_{e_j}-{\bf D}_{[e_j,e_b]})e_0,e_i)\\
=&\,{\bf g}({\bf D}_{e_j}(K_b{}^ce_c)-{\bf D}_{e_b}(K_j{}^ce_c),e_i)-(\Gamma_{jb}{}^c-\Gamma_{bj}{}^c)K_{ci}\\
=&\,e_jK_{bi}+K_b{}^c\Gamma_{jci}-e_bK_{ji}-K_j{}^c\Gamma_{bci}-(\Gamma_{jb}{}^c-\Gamma_{bj}{}^c)K_{ci}
\end{align*}
and
\begin{align*}
R_{aijb}=&\,g((D_{e_a}D_{e_i}-D_{e_i}D_{e_a}-D_{[e_a,e_i]})e_j,e_b)\\
=&\,g(D_{e_a}(\Gamma_{ij}{}^ce_c),e_b)-g(D_{e_i}(\Gamma_{aj}{}^ce_c),e_b)-(\Gamma_{ai}{}^c-\Gamma_{ia}{}^c)g(D_{e_c}e_j,e_b)\\
=&\,e_a\Gamma_{ijb}+\Gamma_{ij}{}^c\Gamma_{acb}-e_i\Gamma_{ajb}-\Gamma_{aj}{}^c\Gamma_{icb}-(\Gamma_{ai}{}^c-\Gamma_{ia}{}^c)\Gamma_{cjb}
\end{align*}
which can be seen to correspond to the asserted formulas by using the antisymmetry of $\Gamma_{ijb}$ in $(j;b)$.
\end{proof}
\begin{lemma}\label{lem:evoleq}
The scalar functions $K_{ij},\Gamma_{ijb},f_i{}^j,f^b{}_j$ satisfy the following evolution equations:
\begin{align}
e_0K_{ij}+\mathrm{tr}KK_{ij}=&-R_{ij}^{{}{(S)}}+{\bf R}_{ij}^{{}{(S)}}, \notag\\
=&\,\frac{1}{2}\bigg[e_i\Gamma^b{}_{jb}-e^b\Gamma_{ijb} \label{e0Kij}
{}{+\Gamma^b{}_i{}^c\Gamma_{cjb}+\Gamma^b{}_b{}^c\Gamma_{ijc}}\\
\notag&+e_j\Gamma^b{}_{ib}-e^b\Gamma_{jib}
{}{+\Gamma^b{}_j{}^c\Gamma_{cib}+\Gamma^b{}_b{}^c\Gamma_{jic}}\bigg]+{\bf R}_{ij}^{{}{(S)}}\\
\notag
e_0\Gamma_{ijb}+{K_i}^c\Gamma_{cjb}=&\,D_jK_{bi}-D_bK_{ji}\\
=&\,e_jK_{bi}-e_bK_{ji}-\Gamma_{jb}{}^cK_{ci}-\Gamma_{ji}{}^cK_{bc}+\Gamma_{bj}{}^cK_{ci}+\Gamma_{bi}{}^cK_{jc} \label{e0Gammaijb} \\
\label{e0fij}e_0f_i{}^j+K_i{}^cf_c{}^j=&\,0\\
\label{e0fijinv}e_0f^b{}_j-K_c{}^bf^c{}_j=&\,0
\end{align}
for all indices $i,j,b=1,2,3$, where 
\begin{align}\label{RicS}
R^{{}{(S)}}_{ij}:=\frac{1}{2}(R_{ij}+R_{ji}),\qquad{\bf R}_{ij}^{{}{(S)}}:=\frac{1}{2}({\bf R}_{ij}+{\bf R}_{ji})
\end{align} 
\end{lemma}
\begin{remark}The Ricci tensor associated to the Levi-Civita connection is always symmetric, and thus $R^{{}{(S)}}_{ij}=R_{ij}$ in this case. However, in order to establish local well-posedness, we will construct the connection from modified equations below and it will no longer hold a priori that $R_{ij}$ or $K_{ij}$ are symmetric, unless we expand the right-hand side of \eqref{e0Kij} in terms of the symmetrised Ricci tensor $R_{ij}^{{}{(S)}}$. In this form, the symmetry of $K_{ij}$ is automatically propagated, provided it is valid initially. 
\end{remark}
\begin{proof}
The propagation condition \eqref{geodgauge} implies the second variation equation
\begin{align}\label{2ndvar}
\notag{\bf R}_{0i0j}=&\, \bg((\bD_{e_0} \bD_{e_i} - \bD_{e_i} \bD_{e_0} -\bD_{[e_0, e_i]}) e_0, e_j ) =\bg (  \bD_{e_0} ( K_{i}{}^c e_c)   -\bD_{( \bD_{e_0} e_i - \bD_{e_i}e_0 ) } e_0, e_j )\\
=&\, e_0 K_{ij}  + K_{i}{}^b \bg ( \bD_{ e_b } e_0, e_j ) 
= e_0K_{ij}+{K_i}^bK_{bj}.
\end{align}
Utilising \eqref{Gauss} we have
\begin{align}\label{R0i0jGauss}
{\bf R}_{0i0j}=-{\bf R}_{0ij0}={\bf R}_{ij}-{\bf R}_{bij}{}^b=
{\bf R}_{ij}-R_{ij}-\text{tr}KK_{ij}+{K_i}^bK_{jb}
\end{align} 
On the other hand, contracting \eqref{Raijb} in $(a;b)$ gives
\begin{align}\label{spatialRij}
-R_{ij}=-R_{bij}{}^b
\notag=&\,e_i\Gamma^b{}_{jb}-e^b\Gamma_{ijb}
+\Gamma^b{}_i{}^c\Gamma_{cjb}-\Gamma_i{}^{bc}\Gamma_{cjb}+\Gamma^b{}_b{}^c\Gamma_{ijc}-\Gamma_i{}^{bc}\Gamma_{bjc}\\
{}{=}&\,{}{e_i\Gamma^b{}_{jb}-e^b\Gamma_{ijb}
+\Gamma^b{}_i{}^c\Gamma_{cjb}+\Gamma^b{}_b{}^c\Gamma_{ijc}}\\
\notag=&\,e_j\Gamma^b{}_{ib}-e^b\Gamma_{jib}
{}{+\Gamma^b{}_j{}^c\Gamma_{cib}+\Gamma^b{}_b{}^c\Gamma_{jic}}=-R_{ji},
\end{align}
where in the last equality we used the symmetry of the Ricci tensor of $g$.
Combining \eqref{2ndvar}, \eqref{R0i0jGauss} and \eqref{spatialRij}, we conclude \eqref{e0Kij}.

By \eqref{geodgauge} and the Codazzi equation \eqref{R0ijb} it follows that
\begin{align}\label{e0Gammaijb2}
\notag e_0\Gamma_{ijb}=&\,{\bf g}({\bf D}_{e_0}{\bf D}_{e_i}e_j,e_b)={\bf R}_{0ijb}+{\bf g}({\bf D}_{e_i}{\bf D}_{e_0}e_j,e_b)+{\bf g}({\bf D}_{[e_0,e_i]}e_j,e_b)\\
=&\,{\bf R}_{0ijb}+{\bf g}\left({\bf D}_{\left( {\bf D}_{e_0}e_i-{\bf D}_{e_i}e_0 \right)}e_j,e_b\right)\\
\notag=&\,D_jK_{bi}-D_bK_{ji}-{K_i}^c\Gamma_{cjb},
\end{align}
which yields \eqref{e0Gammaijb}.

Finally, we have 
\begin{align*}
K_i{}^ce_c={\bf D}_{e_i}e_0={\bf D}_{e_i}e_0-{\bf D}_{e_0}e_i=[e_i,e_0]=[f_i{}^j\partial_j,\partial_t]\qquad\Rightarrow\qquad K_i{}^cf_c{}^j\partial_j=-e_0f_i{}^j\partial_j,
\end{align*}
which implies \eqref{e0fij}. Utilising the relation $f_i{}^jf^b{}_j=\delta^b_i$, we also conclude \eqref{e0fijinv}.
\end{proof}
%
%
%
%
\begin{remark}
Contracting the formula \eqref{spatialRij} and using antisymmetry of $\Gamma_{ijb}$ with respect the last two indices, we notice that the two first order terms combine to give
\begin{align}\label{spatialR}
-R=2e^j\Gamma^b{}_{jb}{}{+\Gamma^{bjc}\Gamma_{cjb}}+\Gamma^b{}_b{}^c\Gamma^j{}_{jc}.
\end{align}
\end{remark}
{}{In the next lemma, we illustrate the structure of the equations \eqref{e0Kij}-\eqref{e0Gammaijb} that we exploit in the local existence argument below, by deriving the main energy identity for $K_{ij},\Gamma_{ijb}$ (at zeroth order). For the moment, we make use of both the Hamiltonian and momentum contraints \eqref{Hamconst}-\eqref{momconst}, i.e., the fact that we have an actual solution to \eqref{EVE}.}
\begin{lemma}\label{lem:loss} 
Let ${\bf g}$ be a solution to the EVE. Then the variables $K_{ij},\Gamma_{ijb}$ satisfy the following identity:
\begin{align}\label{lossid}
\notag&\frac{1}{2}e_0(|K|^2)+\mathrm{tr}K |K|^2+\frac{1}{4}e_0[\Gamma_{ijb}\Gamma^{ijb}]+\frac{1}{2}K_i{}^c\Gamma_{cjb}\Gamma^{ijb}\\
=&\,e_j[K^{ij}\Gamma^b{}_{ib}]-e^i[\mathrm{tr}K\Gamma^b{}_{ib}]-e_b[\Gamma^{ijb}K_{ji}]+\frac{1}{2}\mathrm{tr}K\left[(\mathrm{tr}K)^2-|K|^2{}{+\Gamma^{bic}\Gamma_{cib}+\Gamma^b{}_b{}^c\Gamma^i{}_{ic}}\right]\\
\notag&-\Gamma_j{}^i{}_cK^{cj}\Gamma^b{}_{ib}
-\Gamma_j{}^j{}_cK^{ic}\Gamma^b{}_{ib}+K^{ij}{}{[\Gamma^b{}_j{}^c\Gamma_{cib}+\Gamma^b{}_b{}^c\Gamma_{jic}]}
+{}{\Gamma^{ijb}[\Gamma_{bj}{}^cK_{ci}+\Gamma_{bi}{}^cK_{jc}]},
\end{align}
where $|K|^2=K^{ij}K_{ij}$.
\end{lemma}
\begin{proof}
Multiplying \eqref{e0Gammaijb} by $\Gamma^{ijb}$ and using its antisymmetry in $(j;b)$ gives the identity 
\begin{align}\label{eq:edg}
\frac{1}{4}e_0[\Gamma_{ijb}\Gamma^{ijb}]+\frac{1}{2}{K_i}^c\Gamma_{cjb}\Gamma^{ijb} =-e_b[\Gamma^{ijb}K_{ji}]+K^{ji}e^b\Gamma_{ijb}
+{}{\Gamma^{ijb}[\Gamma_{bj}{}^cK_{ci}+\Gamma_{bi}{}^cK_{jc}].} 
\end{align}
Multiplying \eqref{e0Kij} by $K^{ij}$ and using its symmetry in $(i;j)$, we also have
\begin{align}
\frac{1}{2}e_0(|K|^2)+\mathrm{tr}K |K|^2 
=K^{ij}e_i\Gamma^b{}_{jb}-K^{ij}e^b\Gamma_{ijb}
+K^{ij}{}{[\Gamma^b{}_i{}^c\Gamma_{cjb}+\Gamma^b{}_b{}^c\Gamma_{ijc}]}. \label{eq:edK}
\end{align}
Notice that the second terms on the right-hand sides of \eqref{eq:edg} and \eqref{eq:edK} are exact opposites, hence, canceling out upon summation of the two identities.

We proceed by rewriting the first term on the right-hand side of \eqref{eq:edK}, making use of both constraint equations \eqref{momconst}-\eqref{Hamconst} in the following manner 
\begin{align*}
K^{ij}e_i\Gamma^b{}_{jb}= &\,e_i[K^{ij}\Gamma^b{}_{jb}]- e_i(K^{ij})\Gamma^b{}_{jb} \\
=&\,e_i[K^{ij}\Gamma^b{}_{jb}]-D_iK^{ij}\Gamma^b{}_{jb}-\Gamma_i{}^i{}_cK^{cj}\Gamma^b{}_{jb}-\Gamma_i{}^j{}_cK^{ic}\Gamma^b{}_{jb}\\
\tag{by \eqref{momconst}}=&\,e_i[K^{ij}\Gamma^b{}_{jb}]-e^j\text{tr}K\Gamma^b{}_{jb}
-\Gamma_i{}^i{}_cK^{cj}\Gamma^b{}_{jb}-\Gamma_i{}^j{}_cK^{ic}\Gamma^b{}_{jb}\\
=&\,e_i[K^{ij}\Gamma^b{}_{jb}]-e^j[\mathrm{tr}K\Gamma^b{}_{jb}]+\text{tr}Ke^j\Gamma^b{}_{jb}
-\Gamma_i{}^i{}_cK^{cj}\Gamma^b{}_{jb}-\Gamma_i{}^j{}_cK^{ic}\Gamma^b{}_{jb}\\
=&\,e_i[K^{ij}\Gamma^b{}_{jb}]-e^j[\mathrm{tr}K\Gamma^b{}_{jb}]\\
\tag{by \eqref{spatialR}}&-\frac{1}{2}\text{tr}K\big[R{}{-\Gamma^{bjc}\Gamma_{cjb}}-\Gamma^b{}_b{}^c\Gamma^j{}_{jc}\big]
-\Gamma_i{}^i{}_cK^{cj}\Gamma^b{}_{jb}-\Gamma_i{}^j{}_cK^{ic}\Gamma^b{}_{jb}\\
\tag{by \eqref{Hamconst}}=&\,e_i[K^{ij}\Gamma^b{}_{jb}]-e^j[\mathrm{tr}K\Gamma^b{}_{jb}]+\frac{1}{2}\text{tr}K\big[(\text{tr}K)^2-|K|^2{}{+\Gamma^{bjc}\Gamma_{cjb}+\Gamma^b{}_b{}^c\Gamma^j{}_{jc}}\big]\\
&-\Gamma_i{}^i{}_cK^{cj}\Gamma^b{}_{jb}-\Gamma_i{}^j{}_cK^{ic}\Gamma^b{}_{jb}
\end{align*}
Combining the above identities, we obtain \eqref{lossid}.  
\end{proof}
Although the differential identity \eqref{lossid} provides a way of deriving a priori estimates for $K_{ij},\Gamma_{ijb}$, the equations \eqref{e0Kij}-\eqref{e0Gammaijb} are still not eligible for a local existence argument, because of the heavy use of the constraint equations in the argument. Indeed, in a local existence proof via a Picard iteration scheme, the constraints are no longer valid off of the initial hypersurface $\Sigma_0$. This implies that a structure similar to the one identified in Lemma \ref{lem:loss} is  no longer present, which leads to a loss of derivatives. 

We remedy this problem by adding appropriate multiples of the constraints in the RHS of the evolution equations \eqref{e0Kij}-\eqref{e0Gammaijb}, resulting to the system: 
\begin{align}
\label{e0Kijnew}
e_0K_{ij}+\text{tr}KK_{ij}=&\,\frac{1}{2}\bigg[e_i\Gamma^b{}_{jb}-e^b\Gamma_{ijb}
+\Gamma^b{}_i{}^c\Gamma_{cjb}{}{+\Gamma^b{}_b{}^c\Gamma_{ijc}}
+e_j\Gamma^b{}_{ib}-e^b\Gamma_{jib}
+\Gamma^b{}_j{}^c\Gamma_{cib}{}{+\Gamma^b{}_b{}^c\Gamma_{jic}}\bigg]\\
\notag&-\frac{1}{2}\delta_{ij}\bigg[2e^a\Gamma^b{}_{ab}{}{+\Gamma^{bac}\Gamma_{cab}}
+\Gamma^b{}_b{}^c\Gamma^a{}_{ac}+|K|^2-(\text{tr}K)^2\bigg]\\
\label{e0Gammaijbnew}e_0\Gamma_{ijb}+K_i{}^c\Gamma_{cjb}=&\,e_jK_{bi}-e_bK_{ji}-\Gamma_{jb}{}^cK_{ci}-\Gamma_{ji}{}^cK_{bc}+\Gamma_{bj}{}^cK_{ci}+\Gamma_{bi}{}^cK_{jc}\\
\notag&{}{+\delta_{ib}\bigg[e^cK_{cj}-\Gamma_c{}^{cl}K_{lj}-\Gamma^c{}_j{}^lK_{cl}-e_j\text{tr}K\bigg]}\\
\notag&{}{-\delta_{ij}\bigg[e^cK_{cb}-\Gamma_c{}^{cl}K_{lb}-\Gamma^c{}_b{}^lK_{cl}-e_b\text{tr}K\bigg]}
\end{align}
\begin{remark}\label{rem:addterm}
Contracting \eqref{Gauss} in $(a;b)$, $(i;j)$, contracting \eqref{R0ijb} in $(i;b)$, and utilising \eqref{spatialR}, we notice that the added expressions in the last lines of \eqref{e0Kijnew}-\eqref{e0Gammaijbnew} correspond to
\begin{align*}
&-\frac{1}{2}\delta_{ij}\bigg[2e^a\Gamma^b{}_{ab}+{}{\Gamma^{bac}\Gamma_{cab}}
+\Gamma^b{}_b{}^c\Gamma^a{}_{ac}+|K|^2-(\text{tr}K)^2\bigg]\\
=&\,\frac{1}{2}\delta_{ij}[R-|K|^2+(\text{tr}K)^2]=\frac{1}{2}\delta_{ij}[{\bf R}+2{\bf R}_{00}],\\
&{}{\delta_{ib}\bigg[e^cK_{cj}-\Gamma_c{}^{cl}K_{lj}-\Gamma^c{}_j{}^lK_{cl}-e_j\text{tr}K\bigg]-\delta_{ij}\bigg[e^cK_{cb}-\Gamma_c{}^{cl}K_{lb}-\Gamma^c{}_b{}^lK_{cl}-e_b\text{tr}K\bigg]}\\
=&\,\delta_{ib}\bigg[D^cK_{cj}-e_j\text{tr}K\bigg]-\delta_{ij}\bigg[D^cK_{cb}-e_b\text{tr}K\bigg]=\delta_{ib}{\bf R}_{0j}-\delta_{ij}{\bf R}_{0b},
\end{align*}
which are indeed multiples of the Hamiltonian and momentum constraints \eqref{Hamconst}, \eqref{momconst}.
\end{remark}
By definition of the initial data, cf~Section \ref{subsec:ID}, $K_{ij}=K_{ji}$, $\Gamma_{ijb}=-\Gamma_{ibj}$, $f_i{}^jf^b{}_j=\delta^b_i$, $f^b{}_jf_b{}^i=\delta^i_j$ will be valid initially for any solution and the same can be imposed for any iterate in a Picard iteration scheme. 
\begin{lemma}\label{lem:symm}
A solution $K_{ij},\Gamma_{ijb},f_i{}^j,f^b{}_j$ to \eqref{e0Kijnew}, \eqref{e0Gammaijbnew}, \eqref{e0fij}, \eqref{e0fijinv} satisfies the properties $K_{ij}=K_{ji}$, $\Gamma_{ijb}=-\Gamma_{ibj}$, $f_i{}^jf^b{}_j=\delta^b_i$, $f^b{}_jf_b{}^i=\delta^i_j$, provided they hold true initially.
\end{lemma}
\begin{proof}
The variables $K_{ij}-K_{ji}$, $\Gamma_{ijb}+\Gamma_{ibj}$, $f_i{}^jf^b{}_j-\delta^b_i$, $f^b{}_jf_b{}^i-\delta^i_j$ satisfy the following homogeneous ODE system with trivial initial data: 
\begin{align*}
e_0(K_{ij}-K_{ji})+\text{tr}K(K_{ij}-K_{ji})=0&\\
e_0(\Gamma_{ijb}+\Gamma_{ibj})+K_i{}^c(\Gamma_{cjb}+\Gamma_{cbj})=0&\\
e_0(f_i{}^jf^b{}_j-\delta^b_i)+K_i{}^c(f_c{}^jf^b{}_j-\delta^b_c)-K_c{}^b(f_i{}^jf^c{}_j-\delta^c_i)=0&\\
e_0(f^b{}_jf_b{}^i-\delta^i_j)=0&
\end{align*}
This implies that they must be identically zero.
\end{proof}
{}{
\begin{lemma}\label{lem:hypsymm}
The equations \eqref{e0Kijnew}, \eqref{e0Gammaijbnew} constitute a first order symmetric hyperbolic system.
\end{lemma}
\begin{proof}
It suffices to look at the linearised equations around zero:\footnote{{}{In fact, the system \eqref{hypsymm} corresponds exactly to \eqref{e0Kijnew}-\eqref{e0Gammaijbnew} up to zeroth order terms.}}
\begin{align}\label{hypsymm}
\notag e_0K_{11}=e_3\Gamma_{223}-e_2\Gamma_{323},\qquad
2e_0K_{12}=-e_3\Gamma_{123}-e_3\Gamma_{213}+e_1\Gamma_{323}+e_2\Gamma_{313},\\
\notag e_0K_{22}=e_3\Gamma_{113}-e_1\Gamma_{313},\qquad
2 e_0K_{13}=e_3\Gamma_{212}+e_2\Gamma_{123}
-e_1\Gamma_{223}-e_2\Gamma_{312},\\
 e_0K_{33}=e_2\Gamma_{112}-e_1\Gamma_{212},\qquad
 2 e_0K_{23}=-e_3\Gamma_{112}+e_1\Gamma_{213}
-e_2\Gamma_{113}+e_1\Gamma_{312},\\
\notag e_0\Gamma_{113}=e_3K_{22}-e_2K_{23},\qquad
 e_0\Gamma_{223}=e_3K_{11}-e_1K_{13},\qquad
\notag e_0\Gamma_{123}=-e_3K_{12}+e_2K_{13},\\
\notag e_0\Gamma_{213}=-e_3K_{12}+e_1K_{23},\qquad
 e_0\Gamma_{313}=e_2K_{12}-e_1K_{22},\qquad
 e_0\Gamma_{323}=e_1K_{12}-e_2K_{11},\\
\notag e_0\Gamma_{312}=e_1K_{23}-e_2K_{13},\qquad
e_0\Gamma_{112}=-e_3K_{23}+e_2K_{33},\qquad
e_0\Gamma_{212}=e_3K_{13}-e_1K_{33}
\end{align}
As one can tediously check, \eqref{hypsymm} is symmetric.
\end{proof}

\subsection{The differentiated system}\label{subsec:diffsyst}

In order to derive higher order energy estimates below, we will need to work with differentiated versions of \eqref{e0Kijnew}-\eqref{e0Gammaijbnew}. Moreover, for the boundary value problem (Section \ref{sec:totgeodbd}),  we commute the equations with components of the orthonormal frame, which enables us to use the structure identified in \eqref{hypsymm} to control energies that contain an appropriate number of normal derivatives to the boundary (see Proposition \ref{prop:bdlocex}). 

For this purpose, we consider a multi-index $I$ and the corresponding combination of vector fields $e^I$ among $\{e_\mu\}^3_0$. We will use the following commutation formulas to compute the differentiated equations below:
\begin{align}\label{comm}
[e_i,e_0]=K_i{}^ce_c,\qquad [e_i,e_j]=f^d{}_c(e_if_j{}^c)e_d-f^d{}_c(e_jf_i{}^c)e_d.
\end{align}
We note that \eqref{comm} follows by \eqref{fij} and \eqref{e0fij}. It is important that we do not use relations between the orthonormal frame and its connection coefficients\footnote{As for example, $[e_i,e_j]=\Gamma_{ij}{}^ce_c-\Gamma_{ji}{}^ce_c$.} to compute the commuted equations, since in a local existence argument it is not a prior known for instance that $\Gamma_{ijb}$ are the connection coefficients of $e_1,e_2,e_3$. The fact that the solution to the modified evolution equations \eqref{e0Kijnew}-\eqref{e0Gammaijbnew} gives indeed the connection coefficients of the orthonormal frame $\{e_\mu\}^3_0$, with respect to the Levi-Civita connection of the metric induced by the latter, is shown in Section \ref{sec:EVEsol} together with the vanishing of the Einstein tensor.
\begin{align}
\notag&e_0e^IK_{ij}+e^I(\text{tr}KK_{ij})\\
\label{eq:diff.K}=&\,\frac{1}{2}\bigg[e_ie^I\Gamma^b{}_{jb}-e^be^I\Gamma_{ijb}
+e_je^I\Gamma^b{}_{ib}-e^be^I\Gamma_{jib}-2\delta_{ij}e^ae^I\Gamma^b{}_{ab}\bigg]-[e^I,e_0]K_{ij}\\
\notag&+\frac{1}{2}\bigg[[e^I,e_i]\Gamma^b{}_{jb}-[e^I,e^b]\Gamma_{ijb}
+[e^I,e_j]\Gamma^b{}_{ib}-[e^I,e^b]\Gamma_{jib}-2\delta_{ij}[e^I,e^a]\Gamma^b{}_{ab}\bigg]\\
\notag&+\frac{1}{2}e^I\bigg[\Gamma^b{}_i{}^c\Gamma_{cjb}+\Gamma^b{}_b{}^c\Gamma_{ijc}
+\Gamma^b{}_j{}^c\Gamma_{cib}+\Gamma^b{}_b{}^c\Gamma_{jic}
-\delta_{ij}\big[\Gamma^{bac}\Gamma_{cab}
+\Gamma^b{}_b{}^c\Gamma^a{}_{ac}+|K|^2-(\text{tr}K)^2\big]\bigg],
\\
\notag&e_0e^I\Gamma_{ijb}+e^I(K_i{}^c\Gamma_{cjb})\\
\label{eq:diff.Gamma}=&\,e_je^IK_{bi}-e_be^IK_{ji}
+\delta_{ib}(e^ce^IK_{cj}-e_je^I\text{tr}K)
-\delta_{ij}(e^ce^IK_{cb}-e_be^I\text{tr}K)-[e^I,e_0]\Gamma_{ijb}\\
\notag&+[e^I,e_j]K_{bi}-[e^I,e_b]K_{ji}
+\delta_{ib}([e^I,e^c]K_{cj}-[e^I,e_j]\text{tr}K)
-\delta_{ij}([e^I,e^c]K_{cb}-[e^I,e_b]\text{tr}K)\\
\notag&+e^I\bigg[\Gamma_{bj}{}^cK_{ci}+\Gamma_{bi}{}^cK_{jc}-\Gamma_{jb}{}^cK_{ci}-\Gamma_{ji}{}^cK_{bc}
-\delta_{ib}(\Gamma_c{}^{cl}K_{lj}+\Gamma^c{}_j{}^lK_{cl})
+\delta_{ij}(\Gamma_c{}^{cl}K_{lb}+\Gamma^c{}_b{}^lK_{cl})\bigg]
\end{align}
The differentiated versions of the equations \eqref{e0fij}, \eqref{e0fijinv} read
\begin{align}
\label{eq:diff.f}e_0e^If_i{}^j+K_i{}^ce^If_c{}^j=&-\sum_{I_1\cup I_2=I,\,|I_2|<|I|}e^{I_1}K_i{}^ce^{I_2}f_c{}^j-[e^I,e_0]f_i{}^j
\\
\label{eq:diff.finv}e_0e^If^b{}_j-K_c{}^be^If^c{}_j=&\sum_{I_1\cup I_2=I,\,|I_2|<|I|}e^{I_1}K_c{}^be^{I_2}f^c{}_j-[e^I,e_0]f^b{}_j
\end{align}
\begin{lemma}\label{lem:loss2}
Let $K_{ij},\Gamma_{ijb},f_i{}^j$ be either a solution to \eqref{e0fij},
\eqref{e0Kijnew}, \eqref{e0Gammaijbnew} or an iterative version of these equations, where the frame coefficients $f_i{}^j$ (and hence $e_i=f_i{}^j\partial_j$) are determined by solving \eqref{e0fij} with $K_{ij}$ of the previous step. In the latter case, the first order terms in the RHS of \eqref{e0Kijnew}-\eqref{e0Gammaijbnew} should have $K_{ij},\Gamma_{ijb}$ of the current iterates we're solving for. Then for any a combination of derivatives $e^I$, $K_{ij},\Gamma_{ijb}$ satisfy the following identity:
\begin{align}\label{lossid2}
\notag&\frac{1}{2}e_0(e^IK^{ij}e^IK_{ij})+\frac{1}{4}e_0(e^I\Gamma^{ijb}e^I\Gamma_{ijb})\\
\notag=&\,e_i[e^IK^{ij}e^I\Gamma^b{}_{jb}]-e^b[e^IK^{ij}e^I\Gamma_{ijb}]-e_j[e^I\mathrm{tr}Ke^I\Gamma^{bj}{}_b]\\
&+e^IK^{ij}\bigg[[e^I,e_i]\Gamma^b{}_{jb}-[e^I,e^b]\Gamma_{ijb}-\delta_{ij}[e^I,e^a]\Gamma^b{}_{ab}\bigg]-e^IK^{ij}[e^I,e_0]K_{ij}\\
\notag&+e^I\Gamma^{ijb}\bigg[[e^I,e_j]K_{bi}
+\delta_{ib}([e^I,e^c]K_{cj}-[e^I,e_j]\mathrm{tr}K)\bigg]-\frac{1}{2}e^I\Gamma^{ijb}[e^I,e_0]\Gamma_{ijb}\\
\notag&+e^IK^{ij}e^I\bigg[\Gamma^b{}_i{}^c\Gamma_{cjb}+\Gamma^b{}_b{}^c\Gamma_{ijc}-\mathrm{tr}KK_{ij}
-\frac{1}{2}\delta_{ij}\big[\Gamma^{bac}\Gamma_{cab}
+\Gamma^b{}_b{}^c\Gamma^a{}_{ac}-|K|^2+(\mathrm{tr}K)^2\big]\bigg]\\
\notag&+e^I\Gamma^{ijb}e^I\bigg[\Gamma_{bj}{}^cK_{ci}+\Gamma_{bi}{}^cK_{jc}-\frac{1}{2}K_i{}^c\Gamma_{cjb}-\Gamma_c{}^{cl}K_{lj}-\Gamma^c{}_j{}^lK_{cl}\bigg]
\end{align}
\end{lemma}
\begin{proof}
It follows straightforwardly by multiplying \eqref{eq:diff.K}-\eqref{eq:diff.Gamma} with $e^IK^{ij},\frac{1}{2}e^I\Gamma^{ijb}$ and using Lemma \ref{lem:symm}.
\end{proof}

}

\subsection{Local well-posedness of the reduced equations for the Cauchy problem} \label{se:lwpde}

Since the above equations form a symmetric hyperbolic system, local well-posedness follows from standard arguments. Nonetheless, we provide details below concerning the derivation of higher order energy estimates and the domain of dependence. This will allow us to treat the boundary case by a modification of the present section. 

Define the $H^s(U_t)$ spaces, $U_t\subset \Sigma_t$, as the set of functions satisfying 
\begin{align}\label{Hs}
\|u\|_{H^s(U_t)}^2:=\sum_{|I|\leq s}\int_{U_t}(e^Iu)^2\mathrm{vol}_{U_t}<+\infty,
\end{align}
where $I$ is a multi-index consisting only of spatial indices so that $e^I$ is a combination of $I$ derivatives among $e_1,e_2,e_3$,  and $\mathrm{vol}_{U_t}$ is the intrinsic volume form. One might need more than one orthonormal frame to cover all of $T\Sigma_t$, but we could also consider the corresponding norms restricted to the slicing $U_t$ of a neighbourhood of a point.
In the case where $u$ depends on various spatial indices, we define its $H^s$ norm similarly, summing as well over all indices.
\begin{remark}
The above $H^s$ spaces are equivalent to the usual spaces defined using coordinate derivatives $\partial^I$, provided we have control over the transition coefficients $f_i{}^j,f^b{}_j$.
The use of $e^I$ vector fields is essential for the treatment of the boundary problem in the next section. For this subsection we could have used $\partial^I$ instead. 
\end{remark}
\begin{lemma}\label{lem:Hs}
Let $U_t$ be an open, bounded, subset of $\Sigma_t$ with smooth boundary.
Assume the transition coefficients $f_i{}^j,f^b{}_j$ satisfy the bounds 
\begin{align}\label{controlfij}
\sum_{|I|\leq1}\sum_{i,j,b=1}^3\sup_{t\in[0,T]}(\|e^If_i{}^j\|_{L^\infty(U_t)}+\|e^If^b{}_j\|_{L^\infty(U_t)})\leq D,
\end{align}
for some $T>0$. Then the following Sobolev inequalities hold with respect to the $H^s$ spaces defined above:
\begin{align}\label{Sob}
\|u\|_{L^\infty(U_t)}\leq C\|u\|_{H^2(U_t)},\qquad\|u\|_{L^4(U_t)}\leq C\|u\|_{H^1(U_t)}
\end{align}
for all $t\in[0,T]$, where $C>0$ depends on $U_t$ and $D$.
\end{lemma}
\begin{proof}
It is immediate by invoking the corresponding classical inequalities (for $H^s$ spaces defined via coordinate derivatives) and using \eqref{controlfij}:
\begin{align*}
\|u\|_{L^\infty(U_t)}^2\leq&\, C\int_{U_t}(\partial^2u)^2+(\partial u)^2+u^2\mathrm{vol}_{U_t}\\
=&\,C\int_{U_t}f^4(e^2u)^2+f^2(ef)^2(eu)^2+f^2(e u)^2+u^2\mathrm{vol}_{U_t}\\
\leq& \,C\int_{U_t}D^4(e^2u)^2+D^4(eu)^2+D^2(e u)^2+u^2\mathrm{vol}_{U_t}
\end{align*}
The second inequality is derived similarly.
\end{proof}
\begin{proposition}\label{prop:locex}
The system of reduced equations \eqref{e0Kijnew}, \eqref{e0Gammaijbnew}, \eqref{e0fij}, \eqref{e0fijinv} is well-posed in $L^\infty_tH^s$, $s\ge 3$, with initial data prescribed along the Cauchy hypersurface $\Sigma_0$. 
\end{proposition}
\begin{proof}
We assume that a globally hyperbolic solution exists, in the relevant spaces, and derive a priori energy estimates below. Since the estimates for $f_i{}^j,f^b{}_j$ can be trivially derived using the ODEs \eqref{e0fij}-\eqref{e0fijinv}, we assume we already have control over their $H^s$ norm. Note that the assumption $s\ge3$ is consistent with the pointwise control \eqref{controlfij} of up one derivative of $f_i{}^j,f^{}_j$ via \eqref{Sob}.

{}{Consider the differentiated equations \eqref{eq:diff.K}-\eqref{eq:diff.Gamma}, for a multi-index $I$ of order $|I|\leq s$, over the future domain of dependence\footnote{Since the domain of dependence depends on the spacetime metric, in a Picard iteration the actual region of spacetime is not known until after a solution has been found, but one can enlarge slightly the domain to guarantee that in the end the resulting region includes the true domain of dependence of $U_0$.} of a neighbourhood of a point, $U_0$, foliated by $U_t$, $t\in[0,T]$, for some small $T\ge0$. Using Lemma \ref{lem:loss2}, we obtain the energy inequality:}
\begin{align}\label{enineq}
\notag&\sum_{|I|\leq s}\partial_t\int_{U_t}\frac{1}{2}\sum_{i,j}(e^IK_{ij})^2+\frac{1}{4}\sum_{i,j,b}(e^I\Gamma_{ijb})^2\mathrm{vol}_{U_t}\\
\leq&-\sum_{|I|\leq s}\int_{\partial U_t}\frac{1}{2}\sum_{i,j}(e^IK_{ij})^2+\frac{1}{4}\sum_{i,j,b}(e^I\Gamma_{ijb})^2\mathrm{vol}_{\partial U_t}\\
\tag{$\partial_t\mathrm{vol}_{U_t}=\mathrm{tr}K\mathrm{vol}_{U_t}$}&+\sum_{|I|\leq s}\int_{U_t}\bigg[\frac{1}{2}\sum_{i,j}(e^IK_{ij})^2+\frac{1}{4}\sum_{i,j,b}(e^I\Gamma_{ijb})^2\bigg]\text{tr}K\mathrm{vol}_{U_t}\\
\notag&+\sum_{|I|\leq s}\int_{U_t}e_j[e^IK^{ij}e^I\Gamma^b{}_{ib}]-e^i[e^I\text{tr}Ke^I\Gamma^b{}_{ib}]-e_b[e^I\Gamma^{ijb}e^IK_{ji}]\mathrm{vol}_{U_t}\\
\notag&+\sum_{|I|\leq s}\int_{U_t}\bigg(e^IK^{ij}\bigg[[e^I,e_i]\Gamma^b{}_{jb}-[e^I,e^b]\Gamma_{ijb}-\delta_{ij}[e^I,e^a]\Gamma^b{}_{ab}\bigg]-e^IK^{ij}[e^I,e_0]K_{ij}\\
\notag&+e^I\Gamma^{ijb}\bigg[[e^I,e_j]K_{bi}
+\delta_{ib}([e^I,e^c]K_{cj}-[e^I,e_j]\mathrm{tr}K)\bigg]-\frac{1}{2}e^I\Gamma^{ijb}[e^I,e_0]\Gamma_{ijb}\bigg)\mathrm{vol}_{U_t}\\
\tag{by \eqref{Sob}}&+C\|K\|_{H^s(U_t)}\|\Gamma\|_{H^s(U_t)}^2+C\|K\|_{H^s(U_t)}^3,
\end{align}
for a constant $C>0$, depending on the number of derivatives $s$ and $U_t$.
{}{The first term in the RHS comes from the coarea formula,\footnote{{}{Write $U_t$ as a union of an open set $U_T$ (independent of $t$) and 2D surfaces constituting a variation of $\partial U_t$ in the inward normal direction $N$ to the surfaces. Decomposing $\partial_t=L-N$, we notice that $L$ commutes with the integral, while the $-N$ component gives the additional boundary term above.}} having a negative sign since the null boundary of $\{U_\tau\}_{\tau\in[0,t]}$ is ingoing.}
The last line includes all the terms corresponding to {}{the last two lines in \eqref{lossid2},} which are treated by estimating the lowest order term in $L^\infty$ and using Cauchy-Schwarz.
To bound the terms in the second and third from last lines, we expand the commutators schematically for the two types of terms using \eqref{comm}: 
\begin{align}\label{enineq2}
\bigg|\int_{U_t}e^IK*[e^I,e_0]K\mathrm{vol}_{U_t}\bigg|=&\,\bigg|\sum_{|I_1|+|I_2|=|I|-1}\int_{U_t}e^IK*e^{I_1}K*e^{I_2}eK\mathrm{vol}_{U_t}\bigg|
\overset{\eqref{Sob}}{\leq} C\|K\|_{H^s}^3
\end{align}
and
\begin{align}\label{enineq3}
\bigg|\int_{U_t}e^IK*[e^I,e_i]\Gamma\mathrm{vol}_{U_t}\bigg|=&\,\bigg|\sum_{|I_1|+|I_2|+| I_3|=|I|-1}\int_{U_t}e^IK*e^{I_1}f*e^{I_2}ef*e^{I_3}e\Gamma\mathrm{vol}_{U_t}\bigg|\\
\tag{by \eqref{Sob}}\leq&\, C\|K\|_{H^s}\|f\|^2_{H^s}\|\Gamma\|_{H^s},
\end{align}
where we make use of the inequality $\|u\|_{L^4(U_t)}\leq C\|u\|_{H^1(U_t)}$ only in the last estimate, when $s=3$, $|I_2|=|I_3|=1$, after performing Cauchy-Schwarz twice, otherwise estimating the lowest order term in $L^\infty$.

We may combine \eqref{enineq}-\eqref{enineq3}, integrate in $[0,t]$ and integrate by parts\footnote{The coefficients of the interior terms generated by integrating by parts the terms in the fourth line in \eqref{enineq}, contain first derivatives of $f_i{}^j$ that can be estimated in $L^\infty$.} to obtain the overall integral inequality
\begin{align}\label{enineq4}
\notag\|u\|_{H^s(U_t)}^2\leq &\,\|u\|_{H^s(U_0)}^2+C\int^t_0\|u\|_{H^s(U_\tau)}^3d\tau\\
&-\int^t_0\sum_{|I|\leq s}\int_{\partial U_t}\frac{1}{2}\sum_{i,j}(e^IK_{ij})^2+\frac{1}{4}\sum_{i,j,b}(e^I\Gamma_{ijb})^2\mathrm{vol}_{U_t}d\tau\\
\notag&+\int^t_0\sum_{|I|\leq3}\int_{\partial U_t}\bigg[e^IK^i{}_ce^I\Gamma^b{}_{ib}-e^I\text{tr}Ke^I\Gamma^b{}_{cb}-e^I\Gamma^{ij}{}_ce^IK_{ji}\bigg]N_{\partial U_\tau}^c\mathrm{vol}_{U_t}d\tau
\end{align}
where $u=(\frac{1}{\sqrt{2}}K_{ij},\frac{1}{2}\Gamma_{ijb})_{i,j,b=1,2,3}$ and $N_{\partial U_\tau}$ is the outward unit normal to $U_t$ in $\Sigma_t$. 

It remains to show that the sum of all boundary terms in the last two lines of \eqref{enineq3} has a favourable sign. For notational simplicity in the following computations, we assume\footnote{This is without loss of generally, since $N_{\partial U_\tau}$ can be written as a linear combination of $e_1,e_2,e_3$, where the sum of the squares of the coefficients is $1$. Repeating the argument that follows for each component of $N_{\partial U_\tau}$, leads to the same conclusion by using Cauchy's inequality.} that $e_3=N_{\partial U_\tau}$ and omit $e^I$. The integrands of the boundary terms then read:
\begin{align}
\notag&-\frac{1}{2}K^{ij}K_{ij}-\frac{1}{4}\Gamma^{ijb}\Gamma_{ijb}
+K^i{}_3\Gamma^b{}_{ib}-\text{tr}K\Gamma^b{}_{3b}-\Gamma^{ij}{}_3K_{ji}\\
\tag{$\hat{K}_{AB}:=K_{AB}-\frac{1}{2}\delta_{AB}K_C{}^C$}=&-\frac{1}{2}\left( K_{33} \right)^2-K^A{}_3K_{A3}-\frac{1}{2}\hat{K}^{AB}\hat{K}_{AB}-{}{\frac{1}{4}}(K_C{}^C)^2\\
\notag&-\frac{1}{4}\Gamma^{ABC}\Gamma_{ABC}-\frac{1}{4}\Gamma_3{}^{AB}\Gamma_{3AB}-\frac{1}{2}\Gamma_{33}{}^A\Gamma_{33A}-\frac{1}{2}\Gamma^A{}_3{}^B\Gamma_{A3B}\\
\notag&+K_3{}^A\Gamma^b{}_{Ab}-K_C{}^C\Gamma^B{}_{3B}-\Gamma^{iB}{}_3K_{iB}\\
\label{enineq5}=&-\frac{1}{2} \left( K_{33} \right)^2-K^A{}_3K_{A3}-\frac{1}{2}\hat{K}^{AB}\hat{K}_{AB}-{}{\frac{1}{4}}(K_C{}^C)^2\\
\tag{$\widehat{\Gamma}_{A3B}:=\Gamma_{A3B}-\frac{1}{2}\delta_{AB}\Gamma_{C3}{}^C$}
&-\frac{1}{4}\Gamma^{ABC}\Gamma_{ABC}-\frac{1}{4}\Gamma_3{}^{AB}\Gamma_{3AB}-\frac{1}{2}\Gamma_{33}{}^A\Gamma_{33A}-\frac{1}{2}\widehat{\Gamma}^A{}_3{}^B\widehat{\Gamma}_{A3B}-{}{\frac{1}{4}}(\Gamma^C{}_{3C})^2\\
\notag&+K_3{}^A\Gamma^B{}_{AB}-K_C{}^C\Gamma^B{}_{3B}-\Gamma^{AB}{}_3K_{AB}. 
\end{align}
Rewrite the last line
\begin{align}\label{enineq6}
\notag &K_3{}^A\Gamma^B{}_{AB}-K_C{}^C\Gamma^B{}_{3B}-\Gamma^{AB}{}_3K_{AB}\\
=&\,K_3{}^A\Gamma^B{}_{AB}-K_C{}^C\Gamma^B{}_{3B}+\Gamma^A{}_3{}^BK_{AB}\\
\notag=&\,K_3{}^A\Gamma^B{}_{AB}-\frac{1}{2}K_C{}^C\Gamma^B{}_{3B}+\widehat{\Gamma}^A{}_3{}^B\hat{K}_{AB}\\
\notag \leq&\, K_3{}^AK_{3A}+\frac{1}{4}\Gamma^{BA}{}_B\Gamma^B{}_{AB}+\frac{1}{4}(K_C{}^C)^2+\frac{1}{4}(\Gamma^B{}_{3B})^2
+\frac{1}{2}\widehat{\Gamma}^A{}_3{}^B\widehat{\Gamma}_{A3B}+\frac{1}{2}\hat{K}^{AB}\hat{K}_{AB}
\end{align}
Notice that $\Gamma^{BA}{}_B\Gamma^B{}_{AB}=(\Gamma_{112})^2+(\Gamma_{221})^2$. 

Thus, plugging \eqref{enineq6} into \eqref{enineq5}, we conclude that the sum of all boundary terms has an overall negative sign. 
Therefore, they can be dropped in \eqref{enineq4}, giving  
\begin{align}\label{enineq7}
\|u\|_{H^s(U_t)}^2\leq &\,\|u\|_{H^s(U_0)}^2+C\int^t_0\|u\|_{H^s(U_\tau)}^3d\tau
\end{align}
The preceding estimate can be upgraded to a Picard iteration and a local existence result in a standard way, we omit the details.
\end{proof}
\subsection{Initial data}\label{subsec:ID}

Our initial data $(\Sigma_0,g,K)$ are that of the EVE, i.e.~the induced metric and the second fundamental form on $\Sigma_0$, verifying the constraint equations \eqref{Hamconst}-\eqref{momconst}, with $h,k$ replaced by $g,K$. Given an orthonormal frame $e_1,e_2,e_3$ on $\Sigma_0$ and an abstract coordinate system $(x_1,x_2,x_3)$, the initial data for $f_i^j,f^b{}_j,K_{ij},\Gamma_{ijb}$ are determined in the obvious way, c.f. \eqref{fij}, \eqref{Kij}, \eqref{Gammaijb}. In particular, the functions $\Gamma_{ijb}$ are the connection coefficients associated to $e_1,e_2,e_3$, with respect to the Levi-Civita connection $D$ of $g$, and they are hence anti-symmetric in the indices $j,b$, while $K_{ij}=K_{ji}$, $f_i{}^jf^b{}_j=\delta^i_j$, $f^b{}_jf_b{}^i=\delta^i_j$ on $\Sigma_0$.

\section{Application to totally geodesic boundaries}\label{sec:totgeodbd}

Next, we show that our framework finds an immediate application to the initial boundary value problem, in the case of timelike, totally geodesic boundaries. We demonstrate the well-posedness of the modified reduced system \eqref{e0fij}, \eqref{e0Kijnew}, \eqref{e0Gammaijbnew}, subject to the induced boundary and compatibility conditions (see Lemmas \ref{lem:compcond.angle}, \ref{lem:bdcond}, \ref{lem:compcond}).
{}{We use the notation $\mathcal{T},\mathcal{S}_t$ for the timelike boundary of $\mathcal{M}$ and its cross sections $\Sigma_t\cap\mathcal{T}$, $\mathcal{S}_0=\mathcal{S}$.}

 \begin{figure}[H]
    \begin{center}
   \includegraphics[scale=1.2]{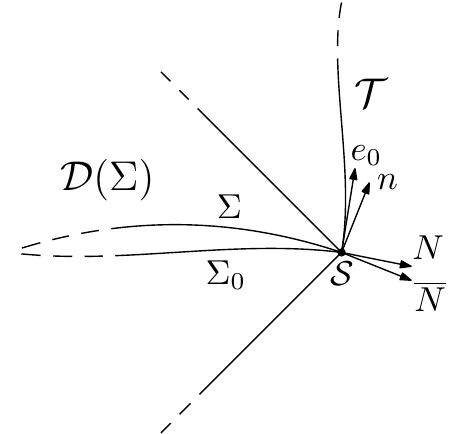}
    \caption{The classical solution in the domain of dependence $\mathcal{D}(\Sigma)$.}
    \label{fig-Sigma}
    \end{center}
  \end{figure}
\subsection{Choosing an orthogonal initial slice to the boundary}\label{subsec:orth.slice}

Let $e_0$ be the normal to $\mathcal{S}$, within $\mathcal{T}$, and $N$ be the outward unit normal to the boundary. The intial slice $\Sigma$ will not in general be orthogonal to the boundary $\mathcal{T}$, but it will have a given angle $\omega:\mathcal{S}\to\mathbb{R}$, that defines the hyperbolic rotation between $(e_0,N)$ and the pair $(n,\overline{N})$, where $n$ is the future unit normal to $\Sigma$ on $\mathcal{S}$ and $\overline{N}$ is the outward unit normal to $\mathcal{S}$, within $\Sigma$, see Figure \ref{fig-Sigma}: 
\begin{align}\label{omega}
\left\{\begin{array}{ll}
e_0=&\,n\cosh\omega-\overline{N}\sinh\omega\\
N=&\,-n\sinh\omega+\overline{N}\cosh\omega
\end{array}\right.,\qquad
\left\{\begin{array}{ll}
n=&\,e_0\cosh\omega+N\sinh\omega\\
\overline{N}=&\,e_0\sinh\omega+N\cosh\omega
\end{array}\right..
\end{align}
However, we may choose an other initial hypersurface, $\Sigma_0$, which is orthogonal to $\mathcal{T}$, contained in the (future and past) domain of dependence of $\Sigma$. Then, we set up the framework presented in the previous section, using the initial data on $\Sigma_0$ that are induced by the classical solution to the EVE in the domain of dependence region $\mathcal{D}(\Sigma)$. Moreover, the compatibility conditions of the latter initial data on $\mathcal{S}$ are induced by those on $\Sigma$, see Lemmas \ref{lem:compcond.angle}, \ref{lem:compcond}.

\subsection{Boundary and compatibility conditions in the geodesic frame} \label{lem:bcgf}

For a spacetime with a timelike boundary $\mathcal{T}$, the $e_0$ geodesics (relative to ${\bf g}$) will not in general remain tangent to $\mathcal{T}$. This makes the geodesic frame \eqref{geodgauge}, as it stands, unsuitable for studying the general boundary value problem. However, for a totally geodesic boundary, the $e_0$ geodesics will indeed foliate a neighbourhood of $\mathcal{T}$. In this case, the second fundamental form of the boundary
\begin{align}\label{chi}
\chi(Y,Z):={\bf g}({\bf D}_{Y}N,Z)=\chi(Z,Y),\qquad Y,Z\perp N,
\end{align}
is identically zero, $\chi\equiv0$. Here $N$ is the outward unit normal to $\mathcal{T}$.

\begin{lemma}[Compatibility conditions with an angle]\label{lem:compcond.angle}
The initial data $(\Sigma,h,k)$ must satisfy the following zeroth order compatibility conditions:
\begin{align}\label{compcond.angle}
\begin{split}
\chi(\overline{X},\overline{Y})=0&\qquad\Longleftrightarrow\qquad k(\overline{X},\overline{Y})\sinh\omega=h(\nabla_{\overline{X}}\overline{N},\overline{Y})\cosh\omega\\
\chi(\overline{X},e_0)=0&\qquad\Longleftrightarrow\;\,\,\qquad\qquad k(\overline{X},\overline{N})=\overline{X}\omega
\end{split}
\end{align}
for all $\overline{X},\overline{Y}\in T\mathcal{S}$, where $\nabla$ is the Levi-Civita connection of $h$. 
The vanishing of $\chi(e_0,e_0)$ is a manifestation of the fact that the boundary $\mathcal{T}$ is ruled by the $e_0$ ${\bf g}$-geodesics, emanating from $\mathcal{S}$, and it imposes compatibility conditions on the gauge of preference (e.g. the lapse of a spacetime foliation).\footnote{Indeed, for a general $t$-foliation with lapse $\Phi$, the acceleration of the unit normal timelike vector field, ${\bf D}_{e_0}e_0$, equals to the spatial gradient $D\log\Phi$. In addition, $\chi(e_0,e_0)={\bf g}({\bf D}_{e_0}N,e_0)=-N\log \Phi$. The latter is automatically zero in our framework, since the lapse is equal to $1$.}
Higher order compatibility conditions can be derived by plugging \eqref{compcond.angle} into the Einstein vacuum equations, see Lemma \ref{lem:compcond}.
\end{lemma}
\begin{proof}
Plugging \eqref{omega} in the defintion \eqref{chi} of $\chi$, we have
\begin{align*}
\chi(\overline{X},\overline{Y})=&\,{\bf g}({\bf D}_{\overline{X}}(-n\sinh\omega+\overline{N}\cosh\omega),\overline{Y})\\
=&-k(\overline{X},\overline{Y})\sinh\omega+h(\nabla_{\overline{X}}\overline{N},\overline{Y})\cosh\omega-h(n,\overline{Y})X\sinh\omega+h(\overline{N},\overline{Y})X\cosh\omega\\
=&-k(\overline{X},\overline{Y})\sinh\omega+h(\nabla_{\overline{X}}\overline{N},\overline{Y})\cosh\omega
\end{align*}
and 
\begin{align*}
\chi(\overline{X},e_0)=&\,{\bf g}({\bf D}_{\overline{X}}(-n\sinh\omega+\overline{N}\cosh\omega),n\cosh\omega-\overline{N}\sinh\omega)\\
=&\,(\cosh\omega)^2 \overline{X}\omega+k(\overline{X},\overline{N})(\sinh\omega)^2-k(\overline{X},\overline{N})(\cosh\omega)^2-(\sinh\omega)^2\overline{X}\omega\\
=&-k(\overline{X},\overline{N})+\overline{X}\omega,
\end{align*}
%
%
which proves the equivalence \eqref{compcond.angle}.
\end{proof}
An essential ingredient in our approach is the use of an adapted frame to the boundary. The existence of such a frame, compatible with the propagation conditon \eqref{geodgauge}, is possible thanks to the vanishing of $\chi$.
\begin{lemma}[Adapted frame to the boundary]\label{lem:adframe}
Let $\Sigma_0$ be orthogonal to $\mathcal{T}$, as in Figure \ref{fig-Sigma}, and let $e_0$ denote its future unit normal. Also, let $e_1,e_2,e_3$ be an orthonormal frame tangent to $\Sigma_0$, such that at the boundary $e_1,e_2\in {}{T\mathcal{S}}$ and $e_3$ coincides with the outward unit normal  $N$. Then the frame verifying \eqref{geodgauge} is adapted to the boundary. In particular, the $e_0$ curves emanating from ${}{\mathcal{S}}$ remain tangent to $\mathcal{T}$ and $e_3=N$ on $\mathcal{T}$.
\end{lemma}
\begin{proof}
Define the tangential, orthonormal frame $\tilde{e}_0,\tilde{e}_1,\tilde{e}_2$ on $\mathcal{T}$ by the condition: 
\begin{align}\label{geodbd}
\slashed{\nabla}_{\tilde{e}_0}\tilde{e}_0=\slashed{\nabla}_{\tilde{e}_0}\tilde{e}_1=\slashed{\nabla}_{\tilde{e}_0}\tilde{e}_2=0,
\end{align}
where $\slashed{\nabla}$ is the covariant connection intrinsic to $\mathcal{T}$. Also, we impose that $\tilde{e}_0=e_0,\tilde{e}_1=e_1,\tilde{e}_2=e_2$ at ${}{\mathcal{S}}$.

Then, $\tilde{e}_0,\tilde{e}_1,\tilde{e}_2,N$ satisfy
\begin{align}\label{geodbd2}
{\bf D}_{\tilde{e}_0}\tilde{e}_0={\bf D}_{\tilde{e}_0}\tilde{e}_1={\bf D}_{\tilde{e}_0}\tilde{e}_2={\bf D}_{\tilde{e}_0}N=0,\qquad \text{on $\mathcal{T}$},
\end{align}
since the second fundamental form of $\mathcal{T}$ vanishes. Hence, the two set of frames $\tilde{e}_0,\tilde{e}_1,\tilde{e}_2,N$ and $e_0,e_1,e_2,e_3$ satisfy the same propagation equation and have the same initial configurations at ${}{\mathcal{S}}$. We arrive at the conclusion that they must coincide.
\end{proof}
\begin{lemma}[Boundary conditions for the orthogonal foliation]\label{lem:bdcond}
For the particular geodesic frame $e_0,e_1,e_2,e_3$ that is adapted to the boundary, as above, the vanishing of $\chi$ induces the following boundary conditions on $K_{ij},\Gamma_{ijb}$:
\begin{align}\label{bdcond}
K_{A3}=K_{3A}=\Gamma_{A3B}=\Gamma_{AB3}=0,
\end{align}
{}{satisfied on $\mathcal{T}$}, for every $A,B=1,2$. 
\end{lemma}
\begin{proof}
The conditions \eqref{bdcond} follow from the relations
\begin{align}\label{bdid}
K_{A3}=K_{3A}={\bf g}({\bf D}_{e_A}e_0,e_3)=-\chi_{0A},\qquad\Gamma_{A3B}=-\Gamma_{AB3}={\bf g}({\bf D}_{e_A}e_3,e_B)=\chi_{AB},\qquad\text{on $\mathcal{T}$},
\end{align}
and the vanishing of $\chi$.
\end{proof}
\begin{lemma}[Compatibility conditions for the orthogonal foliation]\label{lem:compcond}
The initial data of $K_{ij},\Gamma_{ijb},f_i{}^j$ that correspond to the orthonormal frame in Lemma \ref{lem:adframe}, must satisfy corner conditions at ${}{\mathcal{S}}$ to all orders allowed in our energy spaces. The zeroth order conditions are the boundary conditions \eqref{bdcond}, which are induced by \eqref{compcond.angle}, while the first order conditions read:
\begin{align}\label{compcond}
\begin{split}
e_3\Gamma^B{}_{AB}=&\,e^B\Gamma_{3AB}-\Gamma^b{}_b{}^C\Gamma_{3AC},\qquad
e_3K_{22}=-2\Gamma_{31}{}^CK_{1C},\\
e_3K_{11}=&-2\Gamma_{32}{}^CK_{2C},\qquad
e_3K_{12}=\,\Gamma_{31}{}^CK_{2C}+\Gamma_{32}{}^CK_{1C}.
\end{split}
\end{align}
Higher order conditions can be derived by iteratively taking $e_0$ derivatives of \eqref{compcond}, plugging in the evolution equations of the corresponding variables and using the already derived lower order conditions.
\end{lemma}
\begin{proof}
The zeroth order compatibility conditions are derived from the ones with an angle \eqref{compcond.angle}, using the rotation relations \eqref{omega}. We compute on $\mathcal{S}$:
\begin{align*}
K_{A3}=&\,{\bf g}({\bf D}_{e_A}e_0,e_3)={\bf g}({\bf D}_{e_A}(n\cosh\omega-\overline{N}\sinh\omega),-n\sinh\omega+\overline{N}\cosh\omega)\\
=&\,e_A\omega\sinh^2\omega +k(e_A,\overline{N})\cosh^2\omega-k(e_A,\overline{N})\sinh^2\omega-e_A\omega\cosh^2\omega\\
=&-e_A\omega+k(e_A,\overline{N})=0
\end{align*}
and 
\begin{align*}
\Gamma_{A3B}={\bf g}({\bf D}_{e_A}e_3,e_B)={\bf g}({\bf D}_{e_A}(-n\sinh\omega+\overline{N}\cosh\omega),e_B)=-k(e_A,e_B)\sinh\omega+h(\nabla_{e_A}\overline{N},e_B)=0.
\end{align*}
Together with the symmetry/antisymmetry of $K_{3A},\Gamma_{A3B}$, this proves our claim for the zeroth order corner conditions.

Restricting \eqref{e0Kijnew}, \eqref{e0Gammaijbnew}, for $i=3,j=A$ and $i=A,j=3,b=B$ respectively, to the intersection $\mathcal{S}$ and utilising \eqref{bdcond}, we obtain the equations:
\begin{align*}
0=e_0K_{3A}+\text{tr}KK_{3A}=&\,\frac{1}{2}\bigg[e_3\Gamma^b{}_{Ab}-e^b\Gamma_{3Ab}
{}{+\Gamma^b{}_3{}^c\Gamma_{cAb}+\Gamma^b{}_b{}^c\Gamma_{3Ac}}\\
\notag&+e_A\Gamma^b{}_{3b}-e^b\Gamma_{A3b}
{}{+\Gamma^b{}_A{}^c\Gamma_{c3b}+\Gamma^b{}_b{}^c\Gamma_{A3c}}\bigg]\\
=&\frac{1}{2}\bigg[e_3\Gamma^B{}_{AB}-e^B\Gamma_{3AB}
-\Gamma_3{}^{BC}\Gamma_{CAB}+\Gamma^b{}_b{}^C\Gamma_{3AC}-\Gamma_3{}^{BC}\Gamma_{BAC}\bigg]\\
0=e_0\Gamma_{A3B}+K_A{}^c\Gamma_{c3B}=&\,e_3K_{BA}-e_BK_{3A}-\Gamma_{3B}{}^cK_{cA}-\Gamma_{3A}{}^cK_{Bc}+\Gamma_{B3}{}^cK_{cA}+\Gamma_{BA}{}^cK_{3c}\\
&+\delta_{AB}\bigg[{}{e^cK_{c3}-\Gamma_c{}^{cb}K_{b3}-\Gamma^c{}_3{}^bK_{cb}}-e_3\text{tr}K\bigg]\\
=&\,e_3K_{BA}-\Gamma_{3B}{}^CK_{CA}-\Gamma_{3A}{}^CK_{BC}-\delta_{AB}e_3K_C{}^C
\end{align*}
which give the conditions \eqref{compcond}.
\end{proof}

\subsection{Local well-posedness for the initial boundary value problem}\label{se:lwibvp}

Let  $\mathcal{N} \subset \Sigma_0$ be a neighbhorhood of ${}{\mathcal{S}}$. In the following, we consider a Lorentzian manifold with boundary of the form $({}{\mathcal{D},\bf g})$ with ${}{\mathcal{D}}$ foliated by spacelike hypersurfaces $U_t$, ${}{\mathcal{D}}= \bigcup_{t \in [0,T]} U_t$, such that $U_0$ is diffeomorphic to $\mathcal{N}$ and $\partial {}{\mathcal{D}}= U_0 \cup \mathcal{T}{}{\cup \mathcal{H}\cup U_T}$, with $\mathcal{T}$ timelike and $\mathcal{H}$ ingoing null, $T>0$, as depicted in Figure \ref{fig-D}.
 \begin{figure}[H]
    \begin{center}
    \includegraphics[scale=1.2]{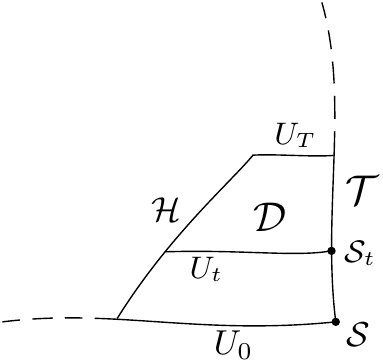}
    \caption{A local domain of dependence region near the boundary.}
    \label{fig-D}
    \end{center}
  \end{figure}
We assume that $({}{\mathcal{D},{\bf g}})$ is globally hyperbolic in the sense of a Lorentzian manifold with timelike boundary \cite{HFS}. In particular, we assume that ${}{\mathcal{D}}$ is such that given $p \in {}{\mathcal{D}}$, $J^{-}(p) \cap J^+ \left( U_0 \cup \mathcal{T} \right)$ is compact, so that all the computations below are well defined. We will prove high order energy estimates on ${}{{\bf g}}$ assuming it solves the Einstein equations with the corresponding initial and boundary data. These a priori estimates can then be upgraded via a Picard iteration to obtain the existence of a solution.

We denote by $\underline e$ any of the derivatives tangential to the boundary $e_0, e_1, e_2$. 
We consider the following modified Sobolev space, denoted $B^s$, which, for a given $s$, contains $[\frac{s}{2}]$ {}{normal} derivatives compared to $s$ tangential derivatives, $[\frac{s}{2}]$ being the integer part of $\frac{s}{2}$. 
%
\begin{definition} \label{def:Bs}
Let $s \in \mathbb{N}$. For any $u \in L^\infty_t L^2(U_t)$, we consider the following energy norm on each slice $U_t$, 
$$\|u\|_{B^s(U_t)}:=\sum_{ | I_1 |+ 2 |I_2| \le s  } \| \underline e^{I_1} e^{I_2}_3u \|_{L^2(U_t)},$$
and the corresponding energy space
\begin{align}\label{Bs}
B^s=\left\{ u \in L^\infty \left( [0,T]; L^2(U_t)\right) \,: \sup_{\,\,\, t \in [0,T] } \mathrm{ess}\sum_{ | I_1 |+ 2 |I_2| \le s  } \| \underline e^{I_1} e^{I_2}_3u \|_{L^2(U_t)}  < + \infty \right\}. 
\end{align}
\end{definition}

\begin{remark}
The need for the $B^s$ 
spaces is dictated by the form \eqref{hypsymm} of the modified ADM system, which only allows the control of roughly half the number of normal derivatives compared with the number of tangential derivatives in $L^2$. Here is where the definition of the norms with respect to $e_\mu$ vector fields becomes particularly useful. Note that we have also included time derivatives in the norms.
\end{remark}
\begin{lemma}\label{lem:Bs}
Let $s\ge6$ and let  
\begin{align*}
\|v\|_{B^s(U_t)}\leq C_0,&&v=K_{ij},\Gamma_{ijb},f_i{}^j,f^b{}_j,
\end{align*}
for all $0\leq t\leq T$.
Changing the order of the tangential and normal derivatives in the definition \eqref{Bs} gives an equivalent norm, up to a constant depending on $C_0$. More precisely, the following inequality holds true
\begin{align}\label{orderBs}
\| e^Iu \|_{L^2(U_t)}\leq C\sum_{|I_1|\leq r}\sum_{|I_2|\leq m} \| \underline e^{I_1} e^{I_2}_3u \|_{L^2(U_t)},
\end{align}
for any $I$ that consists of $r$ tangential and $m$ $e_3$ derivatives, $r+2m\leq s$. The constant $C$ is of the form $C=D+TC_0^2$, where $D>0$ depends only on initial norms.
\end{lemma}
\begin{proof}
We argue by induction in $|I|=r+m$. According to \eqref{comm}, a commutation between a tangential and a normal derivative in $e^Iu$ gives terms of the form 
\begin{align*}
e^{J_1}[\underline{e},e_3]e^{J_2}u=e^{J_1}(K*ee^{J_2}u)+e^{J_1}(f*ef*ee^{J_2}u),\qquad |J_1|+|J_2|=|I|-2
\end{align*}
There are two distinct cases for bounding the $L^2$ norm of the previous RHS. 
\begin{itemize}
\item The derivatives are relatively equally distributed among the corresponding factors, in which case we can bound their $L^2$ norms using the $L^4$ estimate in \eqref{Sob} and the inductive step. 

\item Most derivatives hit one of the factors, in which we can apply \eqref{Sob} to the lowest order factor and use the inductive step.
\end{itemize}
The assumption $s\ge6$ allows for a pointwise bound on the factor $ef$ via \eqref{Sob}. The dependence of the constant $C$ in $D,C_0$ comes from the use of the following basic estimate:
\begin{align*}
\|e^{I'}v\|_{L^2(U_t)}^2\overset{C-S}{\leq}\|e^{I'}v\|_{L^2(U_0)}^2+\int^t_02\|e^{I'}v\|_{L^2(U_\tau)}\|e_0e^{I'}v\|_{L^2(U_\tau)}d\tau\leq D+TC_0^2
\end{align*}
to the terms $v=(K_{ij},\Gamma_{ijb},f_i{}^j,f^b{}_j)$ that have one less than maximum number of tangential derivatives.
\end{proof}

\begin{proposition}\label{prop:bdlocex}
Under the boundary and compatibility conditions 
\eqref{bdcond}, \eqref{compcond}, the system \eqref{e0Kijnew}, \eqref{e0Gammaijbnew}, \eqref{e0fij}, \eqref{e0fijinv} is locally well-posed in   $L^\infty_t B^s(U_t)$, for $s\ge7$. 
\end{proposition}
\begin{proof}
We prove an energy estimate in the space $ L^\infty{}{(}[0, T]; B^s)$, for $s\ge7$ and $T$ sufficiently small. To this end, we proceed by a bootstrap argument and assume that we have a smooth solution $u$ on $[0,T]$, for some $T > 0$, satisfying
\begin{align}\label{bootstrap} 
\qquad\|u\|_{B^s {(U_t)}} \le {}{C_0},
\end{align}
for all $0 \le t \le T$. {}{The slices $U_t$ correspond now to the neighbourhood of a point at the boundary $\Sigma_0\cap\mathcal{T}$.}
We will upgrade this type of non-quantitative estimate into a quantitative one depending only on the initial data. This is the kind of estimate that is required to then prove existence and uniqueness of solutions via a Picard iteration scheme. {}{As for the solution to the linear problem required at each step in the iteration, this follows by a duality argument, using the symmetry of the system \eqref{hypsymm}, combined with the fact that the boundary terms in a usual energy argument for \eqref{e0Kijnew}-\eqref{e0Gammaijbnew} vanish by virtue of the boundary conditions \eqref{bdcond}, see \eqref{bdterms} below. This implies that the dual system has the same form, satisfying the same homogeneous Dirichlet boundary conditions, as in \eqref{bdcond}, for the corresponding dual variables.}

{}{\pfstep{Step~1: Estimates for tangential derivatives}} Consider $I$ a multi-index of order $|I|\le s$, such that $e^I=\underline e^I$ does not contain any $e_3$ derivative. We repeat the energy argument of Proposition \ref{prop:locex}, where we {}{use the differentiated equations \eqref{eq:diff.K}-\eqref{eq:diff.Gamma} for $e^I=\underline e^I$.} Note that by the above bootstrap assumption and since $s\ge 7$, the Sobolev inequalities used to control the $L^\infty,L^4$ norms of certain terms are still applicable. In particular, the error terms generated by the various integration by parts (due to $e_i$ not being Killing) can be controlled in this way and then absorbed by choosing $T$ sufficiently small depending only on the norm of the initial data. 

We examine now the arising $\mathcal{T}$-boundary terms in the energy argument. Going back to the fourth line of \eqref{enineq}, we notice that these terms are 
\begin{align}\label{bdterms}
&\notag\int_{{}{\mathcal{S}_t}} \underline e^IK_3{}^j \underline e^I\Gamma^b{}_{jb}-\underline e^I\text{tr}K\underline e^I\Gamma^b{}_{3b}-\underline e^I\Gamma^{ij}{}_3\underline e^IK_{ij}\mathrm{vol}_{{}{\mathcal{S}_t}}\\
=&\int_{{}{\mathcal{S}_t}}\underline e^IK_{33}\underline e^I\Gamma^B{}_{3B}+\underline e^IK_3{}^A\underline e^I\Gamma^B{}_{AB}-\underline e^I\text{tr}K\underline e^I\Gamma^B{}_{3B}-\underline e^I\Gamma^{3B}{}_3e^IK_{3B} -\underline e^I\Gamma^{AB}{}_3\underline e^IK_{AB} \mathrm{vol}_{{}{\mathcal{S}_t}},
\end{align}
for $|I|\leq s$. Since $\underline e^I$ is tangential, we infer by \eqref{bdcond} that all boundary terms vanish.

Since we commute only with tangential derivatives, the error terms corresponding to \eqref{enineq2}-\eqref{enineq3} take the form 
\begin{align*}
\sum_{|I_1|+|I_2|=|I|-1}\int_{U_t}\underline{e}^IK*\underline{e}^{I_1}K*\underline{e}^{I_2}eK\mathrm{vol}_{U_t}+\sum_{|I_1|+|I_2|+| I_3|=|I|-1}\int_{U_t}\underline e^IK*\underline e^{I_1}f*\underline e^{I_2}ef*\underline e^{I_3}e\Gamma\mathrm{vol}_{U_t}
\end{align*}
Thus, they contain at most one $e_3$ derivative and can be handled using
\eqref{Sob} (since we include at least three normal derivatives
in our $B^s$ norms, $s\ge7$) and the bootstrap assumption. We make use of the $L^4$ estimate only for the second term, in the case $|I_1|=0$, where both $e=e_3$.

\pfstep{Step~2: Consequences of the bootstrap assumption} First, we note that a standard energy argument for \eqref{eq:diff.f}-\eqref{eq:diff.finv} gives the desired estimate for the part of the norm involving $f_i{}^j,f^b{}_j$, making use only of the bootstrap assumption:
\begin{align}\label{ind:fij}
\sum_{i,b,j=1}^3\sup_{t\in[0,T]}(\|\underline{e}^{I_1} e_3^{I_2} f_i{}^j \|_{L^2(U_t)}+\|\underline{e}^{I_1} e_3^{I_2} f^b{}_j \|_{L^2(U_t)})\leq D_f,
\end{align}
where $D_f$ depends on initial norms and $C_0^2T$. The latter can be made smaller than a universal constant by taking $T>0$ sufficiently small. 

Moreover, any term having a less than top order number of tangential derivatives can also be bounded in $L^2$ by using only the bootstrap assumption in the following manner:
\begin{align}
\notag\|\underline e^{I_1'}e_3^{I_2}u\|^2_{L^2(U_t)}\leq&\,\|\underline e^{I_1'}e_3^{I_2}u\|^2_{L^2(U_0)}+\int^t_02\|\underline e^{I_1'}e_3^{I_2}u\|_{L^2(U_\tau)}\|\partial_\tau \underline e^{I_1'}e_3^{I_2}u\|_{L^2(U_\tau)}d\tau\\
\tag{$|I'_1|<s-2|I_2|$}\leq&\,\|\underline e^{I_1'}e_3^{I_2}u\|^2_{L^2(U_0)}+TC_0^2\\
\label{lessthantop}\leq&\,D_{\mathrm{low}}
\end{align}
for some $T>0$ sufficiently small, where $D_{\mathrm{low}}$ denotes a constant depending on the initial $L^2$ norms of $\partial^{I_1'}\partial^{I_2}u,\partial_t\partial^{I_1'}\partial^{I_2}u$. 

Thus, matters are reduced to estimating the top order $K_{ij},\Gamma_{ijb}$ terms.

{}{\pfstep{Step~3: Induction for the normal derivatives}} To complete the energy argument, we must estimate the norms $\| \underline{e}^{I_1} e_3^{I_2} u \|_{L^2}$, with $ |I_1| +2 |I_2| = s$, $u=K,\Gamma$. 

We proceed by induction in $|I_2|$. {\it Step 1} above shows in particular that one can deal with $|I_2|=0$. Let $m\le{}{[\frac{s}{2}]-1}$ and assume that, for all multi-indices $I_1$ and $I_2$ verifying
$ |I_1| +2 |I_2| = s $, $|I_2| \le m$, we have a bound of the form 
\begin{equation}
{}{\sup_{t\in[0,T]}}\| \underline e^{I_1} e^{I_2}_3 u \|_{L^2(U_t)} \le {}{D_m},  \label{ind:h1}
\end{equation}
{}{for some $T>0$ sufficiently small,} where ${}{D_m}$ depends only on the {}{$B^s$-norm of the} initial data of {}{$(K_{ij},\Gamma_{ijb},f_i{}^j,f^b{}_j)$, the number of derivatives $s,m$, and is independent of $C_0$. We will derive \eqref{ind:h1} for $|I_2|=m+1$}.  

Let us split the variables $K_{ij},\Gamma_{ijb}$ into two sets:
\begin{enumerate}
\item the good set
\begin{equation}
\label{def:gs}
\mathcal{G}=\left \{ K_{11},K_{12},K_{22}, K_{31},K_{32},\Gamma_{112},\Gamma_{212}, \Gamma_{113}, \Gamma_{223}, \Gamma_{123}+\Gamma_{213} \right\},
\end{equation}
\item the bad set
\begin{equation}
\label{def:bs}
\mathcal{B}=\left \{ K_{33}, \Gamma_{313},\Gamma_{323}, \Gamma_{312}, \Gamma_{123}-\Gamma_{213}\right\}.
\end{equation}
\end{enumerate}
{}{According to the composition of the system \eqref{e0Kijnew}-\eqref{e0Gammaijbnew} identified in \eqref{hypsymm}, we have the following two schematic types of equations
\begin{align}
\label{e0PsiG}e_0\Psi_{\mathcal{G}}=&\,e_3\Psi_{\mathcal{G}}+e_A\Psi_{\mathcal{B}}+\Psi*\Psi,\\
\label{e0PsiB}e_0\Psi_{\mathcal{B}}=&\,e_A\Psi_{\mathcal{G}}+\Psi*\Psi,
\end{align}
satisfied by $\Psi_{\mathcal{G}}\in\mathcal{G},\Psi_{\mathcal{B}}\in\mathcal{B},\Psi\in\mathcal{G}\cup\mathcal{B}$, where we note that $e_3\Psi_{\mathcal{G}}$ corresponds to a single representative of the good set, whereas $e_A\Psi_{\mathcal{B}},e_A\Psi_{\mathcal{G}}$ could be an algebraic combination of more than one terms from the corresponding sets.

Starting with \eqref{e0PsiG}, we differentiate the equation in $\underline{e}^{I_1}e_3^{I_2'}$, $|I_1|+2|I_2'|=s-1$, $|I_2'|=m$:
\begin{align}\label{diff.PsiG}
\underline{e}^{I_1}e_3^{I_2'}e_3\Psi_{\mathcal{G}}=\underline{e}^{I_1}e_3^{I_2'}e_0\Psi_{\mathcal{G}}-\underline e^{I_1}e_3^{I_2'}e_A\Psi_{\mathcal{B}}-\underline{e}^{I_1}e_3^{I_2'}(\Psi*\Psi)
\end{align}
Note that the derivatives acting on $\Psi_{\mathcal{G}}$ in the LHS contain one more tangential derivative than what we need for the desired estimate on $\Psi_{\mathcal{G}}$.\footnote{We called $\mathcal{G}$ the good set, because its variables have in fact additional regularity than required.} It is necessary to include this extra tangential derivative in order to infer the bound on $\Psi_{\mathcal{B}}$ directly below.

The first two terms in the RHS of \eqref{diff.PsiG} are at the level of the inductive assumption \eqref{ind:h1}. To bound the $L^2$ norm of the third term we employ \eqref{Sob} as follows:
\begin{align}\label{ind:nonlin}
\notag\|\underline{e}^{I_1}e_3^{I_2'}(\Psi*\Psi)\|_{L^2(U_t)}\leq&\, \|\Psi\underline{e}^{I_1}e_3^{I_2'}\Psi\|_{L^2(U_t)}+\sum_{|J_2|,|L_2|<m}\|\underline{e}^{J_1}e_3^{J_2}\Psi*\underline{e}^{L_1}e_3^{L_2}\Psi\|_{L^2(U_t)}\\
\overset{C-S}{\leq}&\,\|\Psi\|_{H^2(U_t)}\|\underline{e}^{I_1}e_3^{I_2'}\Psi\|_{L^2(U_t)}+\sum_{|J_2|,|L_2|<m}\|\underline{e}^{J_1}e_3^{J_2}\Psi\|_{H^1(U_t)}\|\underline{e}^{L_1}e_3^{L_2}\Psi\|_{H^1(U_t)}\\
\tag{by {\it Step 2}}\leq&\,D_{\mathrm{low}}^2
\end{align}
This implies an $L^2$ estimate for $\underline e^{I_1}e_3^{I_2}\Psi_{\mathcal{G}}$, $|I_1|+2|I_2|=s+1$, $|I_2|=m+1$, in accordance with \eqref{ind:h1} (choosing $D_{m+1}\ge D_{\mathrm{low}}^2+D_m$). As we remarked above, the previous term contains one tangential derivative more than required. The desired estimate for $|I_1|+2|I_2|=s$, $|I_2|=m+1$ is in fact simpler.

For $\Psi_{\mathcal{B}}$ we apply $\underline{e}^{I_1}e_3^{I_2}$, $|I_1|+2|I_2|=s$, $|I_2|=m+1$ to \eqref{e0PsiB}:
\begin{align}
\label{diff.PsiB}\underline{e}^{I_1}e_3^{I_2}e_0\Psi_{\mathcal{B}}=\underline{e}^{I_1}e_3^{I_2}e_A\Psi_{\mathcal{G}}+\underline{e}^{I_1}e_3^{I_2}(\Psi*\Psi),
\end{align}
The first term in the RHS has just being controlled in $L^2$ (cf. Lemma \ref{lem:Bs}). The argument for the second term is as in \eqref{ind:nonlin}, only now the final RHS becomes 
\begin{align}\label{ind:nonlin2}
\|\underline{e}^{I_1}e_3^{I_2}(\Psi*\Psi)\|_{L^2(U_t)}\leq D_{\mathrm{low}}^2+ D_{\mathrm{low}} \|\underline{e}^{I_1}e_3^{I_2}\Psi\|_{L^2(U_t)}\leq D_{\mathrm{low}}D_{m+1}+D_{\mathrm{low}} \|\underline{e}^{I_1}e_3^{I_2}\Psi_{\mathcal{B}}\|_{L^2(U_t)}
\end{align}
Thus, combining with Lemma \ref{lem:Bs}, we have the bound 
\begin{align}\label{ind:PsiB}
\|e_0\underline{e}^{I_1}e_3^{I_2}\Psi_{\mathcal{B}}\|_{L^2(U_t)}\leq D_{\mathrm{low}}D_{m+1}+D_{\mathrm{low}} \|\underline{e}^{I_1}e_3^{I_2}\Psi_{\mathcal{B}}\|_{L^2(U_t)}
\end{align}
The first line in \eqref{lessthantop} then gives the estimate
\begin{align}\label{ind:PsiB2}
\sum_{\Psi_{\mathcal{B}}\in\mathcal{B}}\|\underline{e}^{I_1}e_3^{I_2}\Psi_{\mathcal{B}}\|_{L^2(U_t)}^2\leq D+\sum_{\Psi_{\mathcal{B}}\in\mathcal{B}}\int^t_0\|\underline{e}^{I_1}e_3^{I_2}\Psi_{\mathcal{B}}\|_{L^2(U_\tau)}\big(D_{\mathrm{low}}D_{m+1}+D_{\mathrm{low}} \|\underline{e}^{I_1}e_3^{I_2}\Psi_{\mathcal{B}}\|_{L^2(U_\tau)}\big)d\tau
\end{align}
Employing Gronwall's inequality, we obtain the desired estimate for $\Psi_{\mathcal{B}}$ by taking $T>0$ sufficiently small. This completes the proof of the proposition.}
\end{proof}
\section{A solution to the EVE}\label{sec:EVEsol}

In this section we show that the solution of the modified reduced system, with initial data as in Section \ref{subsec:ID}, either for the standard Cauchy problem or for the boundary value problem, subject to the conditions in Lemmas \ref{lem:bdcond}, \ref{lem:compcond}, is in fact a solution to the EVE, see Proposition \ref{prop:EVEsol} and the conclusion in Section \ref{subsec:finstep}. Thus, completing the proofs of Theorems \ref{thmA}, \ref{thmB}.

\subsection{The geometry of a solution to the reduced equations}  \label{se:bictc}

Having solved \eqref{e0Kijnew}, \eqref{e0Gammaijbnew}, \eqref{e0fij}, \eqref{e0fijinv} for $K_{ij},\Gamma_{ijb},f_i{}^j$, $f^b{}_j$, we declare that $e_0=\partial_t$, together with 
$e_1,e_2,e_3$ given by \eqref{fij}, constitute an orthonormal frame. This completely determines the spacetime metric ${\bf g}$, which splits in the form \eqref{metric}. We then need to verify that the variables $K_{ij},\Gamma_{ijb}$ are indeed the second fundamental form of the $t$-slices and spatial connection coefficients of the orthonormal frame we have just defined, with respect to the Levi-Civita connection ${\bf D}$ of ${\bf g}$. In fact, this must be derived at the same time with the vanishing of the spacetime Ricci tensor, confirming that the solution of the reduced system is in fact a solution of the EVE.

For this purporse, we define the connection ${\bf \widetilde{D}}$ by the relations
\begin{align}\label{Ddef}
{\bf \widetilde{D}}_{e_0}e_\mu=0,\qquad {\bf \widetilde{D}}_{e_i}e_0=K_i{}^je_j,\qquad {\bf \widetilde{D}}_{e_i}e_j=\Gamma_{ij}{}^be_b+K_{ij}e_0
\end{align} 
and denote the projection of ${\bf \widetilde{D}}$ onto the span of $e_1,e_2,e_3$ by $\widetilde{D}$. Let 
\begin{align}\label{curvtilde}
{}{{\bf \widetilde{R}}_{\alpha\beta\mu}{}^\nu e_\nu:=({\bf \widetilde{D}}_{e_\alpha}{\bf \widetilde{D}}_{e_\beta}-{\bf \widetilde{D}}_{e_\beta}{\bf \widetilde{D}}_{e_\alpha}-{\bf \widetilde{D}}_{[e_\alpha,e_\beta]})e_\mu=-{\bf \widetilde{R}}_{\beta\alpha\mu}{}^\nu e_\nu,}\qquad {\bf \widetilde{R}}_{\beta\mu}={\bf \widetilde{R}}_{\alpha \beta\mu}{}^\alpha,\qquad {\bf\widetilde{R}}={\bf\widetilde{R}}_\mu{}^\mu 
\end{align}
be the Riemann, Ricci, and scalar curvatures of ${\bf \widetilde{D}}$; the curvatures $\widetilde{R}_{aijb},\widetilde{R}_{ij},\widetilde{R}$ associated to $\widetilde{D}$ are defined similarly. For notational simplicity, we will also use in certain places below the convention $\Gamma_{\alpha\beta\nu}:={\bf g}(\widetilde{\bf D}_{e_\alpha}e_\beta,e_\nu){}{=-\Gamma_{\alpha\nu\beta}}$, {}{despite the fact that we have used $\Gamma$ so far to denote only spatial connection coefficients. In particular, with this convention $\Gamma_{i0j}=-\Gamma_{ij0}=K_{ij},\Gamma_{0\alpha\beta}=0$.}

Define the torsion of ${\bf \widetilde{D}}$:
\begin{align}\label{torsion}
C_{\alpha\mu\nu}={\bf g}([e_\alpha,e_\mu]-{\bf \widetilde{D}}_{e_\alpha}e_\mu +{\bf \widetilde{D}}_{e_\mu}e_\alpha,e_\nu)=-C_{\mu\alpha\nu}
\end{align}
Note that ${\bf \widetilde{D}}$ is not a priori torsion-free, however, it annihilates the metric ${\bf g}$. 
\begin{lemma}\label{lem:Dtilde}
The connection ${\bf \widetilde{D}}$ is compatible with ${\bf g}$, ${\bf \widetilde{D}}{\bf g}=0$, while its curvature and torsion tensors satisfy:
\begin{align}
\label{Cijb}C_{ijb}=&\,f^b{}_le_if_j{}^l-f^b{}_le_jf_i{}^l-\Gamma_{ijb}+\Gamma_{jib}=-C_{jib},\qquad C_{\alpha\beta0}=C_{0ij}=C_{i0j}=0,\\
\label{fBian}0=&\,{\bf \widetilde{R}}_{\alpha\beta\mu\nu}+{\bf \widetilde{R}}_{\beta\mu\alpha\nu}+{\bf \widetilde{R}}_{\mu\alpha\beta\nu}
+{\bf \widetilde{D}}_\mu C_{\alpha\beta\nu}+{\bf \widetilde{D}}_\alpha C_{\beta\mu\nu}+{\bf \widetilde{D}}_\beta C_{\mu\alpha\nu}\\
\notag&{}{+C_{\alpha\beta}{}^l C_{l\mu\nu}+C_{\mu\alpha}{}^l C_{l\beta\nu}+C_{\beta\mu}{}^l C_{l\alpha\nu}}\\
\label{antisRiem}{\bf \widetilde{R}}_{\alpha\beta\mu\nu}=&-{\bf\widetilde{R}}_{\alpha\beta\nu\mu},\qquad \widetilde{R}_{aijb}=-\widetilde{R}_{aibj}\\
\label{sBian}0=&\,{\bf \widetilde{D}}_\mu{\bf \widetilde{R}}_{\alpha\beta\gamma\delta}
+{\bf \widetilde{D}}_\alpha{\bf \widetilde{R}}_{\beta\mu\gamma\delta}
+{\bf \widetilde{D}}_\beta{\bf \widetilde{R}}_{\mu\alpha\gamma\delta}-C_{\mu\alpha l}{\bf \widetilde{R}}^l{}_{\beta\gamma\delta}-C_{\alpha\beta l}{\bf \widetilde{R}}^l{}_{\mu\gamma\delta}-C_{\beta\mu l}{\bf \widetilde{R}}^l{}_{\alpha\gamma\delta}
\end{align}

Moreover, the Gauss and Codazzi equations in Lemma \ref{lem:GaussCod} become:
\begin{align}
\label{Gausstilde}
{\bf \widetilde{R}}_{aijb}=&\,\widetilde{R}_{aijb}+K_{ab}K_{ij}-K_{aj}K_{ib},\\
\label{Codtilde}{\bf \widetilde{R}}_{jb0i}=&\,\widetilde{D}_jK_{bi}-\widetilde{D}_bK_{ji}-C_{jb}{}^lK_{li}.
\end{align}
\end{lemma}
\begin{proof}
The compatibility of ${\bf \widetilde{D}}$ with ${\bf g}$ is equivalent to:
\begin{align*}
({\bf\widetilde{D}}_\alpha{\bf g})_{\mu\nu}=0\qquad\Longleftrightarrow\qquad {\bf g}({\bf \widetilde{D}}_{e_\alpha}e_\mu,e_\nu)+{\bf g}(e_\mu,{\bf \widetilde{D}}_{e_\alpha}e_\nu)=0
\end{align*}
Hence, it follows from the antisymmetry of $\Gamma_{ijb}=-\Gamma_{ibj}$, see Lemma \ref{lem:symm}, and the definition \eqref{Ddef}. Therefore, $\widetilde{D}$ is also compatible with $g$. By definition \eqref{curvtilde}, this also implies the antisymmetry of the curvatures ${\bf\widetilde{R}}_{\alpha\beta\mu\nu},\widetilde{R}_{aijb}$ with respect to the last two indices.

{}{By \eqref{comm}} and \eqref{Ddef} we derive the identities
\begin{align*}
[e_0,e_i]-{\bf \widetilde{D}}_{e_0}e_i+{\bf \widetilde{D}}_{e_i}e_0=&\,{}{-K_i{}^ce_c+K_i{}^ce_c=0},\\
[e_i,e_j]-{\bf \widetilde{D}}_{e_i}e_j+{\bf \widetilde{D}}_{e_j}e_i=&\,f^b{}_le_if_j{}^le_b-f^b{}_le_jf_i{}^le_b-(\Gamma_{ij}{}^b-\Gamma_{ji}{}^b)e_b+(K_{ji}-K_{ij})e_0,
\end{align*}
which yield \eqref{Cijb}, thanks to the symmetry of $K_{ij}$ (Lemma \ref{lem:symm}). 

Next, we derive the first Bianchi identity \eqref{fBian} using the definitions \eqref{curvtilde}, \eqref{torsion}:
\begin{align*}
&{\bf \widetilde{R}}_{\alpha\beta\mu\nu}+{\bf \widetilde{R}}_{\beta\mu\alpha\nu}+{\bf \widetilde{R}}_{\mu\alpha\beta\nu}\\
=&\,{\bf g}(({\bf \widetilde{D}}_{e_\alpha}{\bf \widetilde{D}}_{e_\beta}-{\bf \widetilde{D}}_{e_\beta}{\bf \widetilde{D}}_{e_\alpha}-{\bf \widetilde{D}}_{[e_\alpha,e_\beta]})e_\mu,e_\nu)
+{\bf g}(({\bf \widetilde{D}}_{e_\beta}{\bf \widetilde{D}}_{e_\mu}-{\bf \widetilde{D}}_{e_\mu}{\bf \widetilde{D}}_{e_\beta}-{\bf \widetilde{D}}_{[e_\beta,e_\mu]})e_\alpha,e_\nu)\\
&+{\bf g}(({\bf \widetilde{D}}_{e_\mu}{\bf \widetilde{D}}_{e_\alpha}-{\bf \widetilde{D}}_{e_\alpha}{\bf \widetilde{D}}_{e_\mu}-{\bf \widetilde{D}}_{[e_\mu,e_\alpha]})e_\beta,e_\nu)\\
=&\,{\bf g}({\bf \widetilde{D}}_{e_\alpha}([e_\beta,e_\mu]-C_{\beta\mu}{}^{l} e_{l}),e_\nu)+{\bf g}({\bf \widetilde{D}}_{e_\beta}([e_\mu,e_\alpha]-C_{\mu\alpha}{}^{l} e_{l}),e_\nu)+{\bf g}({\bf \widetilde{D}}_{e_\mu}([e_\alpha,e_\beta]-C_{\alpha\beta}{}^{l} e_{l}),e_\nu)\\
&-{\bf g}({\bf \widetilde{D}}_{[e_\alpha,e_\beta]}e_\mu,e_\nu)-{\bf g}({\bf \widetilde{D}}_{[e_\beta,e_\mu]}e_\alpha,e_\nu)-{\bf g}({\bf \widetilde{D}}_{[e_\mu,e_\alpha]}e_\beta,e_\nu)\\
=&\,{\bf g}({\bf \widetilde{D}}_{e_\alpha}([e_\beta,e_\mu])-{\bf \widetilde{D}}_{[e_\beta,e_\mu]}e_\alpha,e_\nu) + \,{\bf g}({\bf \widetilde{D}}_{e_\beta}([e_\mu,e_\alpha])-{\bf \widetilde{D}}_{[e_\mu,e_\alpha]}e_\beta,e_\nu) + \,{\bf g}({\bf \widetilde{D}}_{e_\mu}([e_\alpha,e_\beta])-{\bf \widetilde{D}}_{[e_\alpha,e_\beta]}e_\mu,e_\nu) \\
&\,- {\bf g}({\bf \widetilde{D}}_{e_\alpha}(C_{\beta\mu}{}^{l} e_{l}),e_\nu)  - {\bf g}({\bf \widetilde{D}}_{e_\beta}(C_{\mu \alpha}{}^{l} e_{l}),e_\nu)- {\bf g}({\bf \widetilde{D}}_{e_\mu}(C_{\alpha\beta}{}^{l} e_{l}),e_\nu)\\
=&\,[e_\mu,[e_\alpha,e_\beta]]+[e_\alpha,[e_\beta,e_\mu]]+[e_\beta,[e_\mu,e_\alpha]]-C_{\mu{l}\nu}{\bf g}(e^{l},[e_\alpha,e_\beta])-C_{\alpha{l}\nu}{\bf g}(e^{l},[e_\beta,e_\mu])\\
&-C_{\beta{l}\nu}{\bf g}(e^{l},[e_\mu,e_\alpha])-e_\alpha C_{\beta\mu\nu}-C_{\beta\mu}{}^{l}\Gamma_{\alpha{l}\nu}-e_\beta C_{\mu\alpha\nu}-C_{\mu\alpha}{}^{l}\Gamma_{\beta{l}\nu}-e_\mu C_{\alpha\beta\nu}-C_{\alpha\beta}{}^{l}\Gamma_{\mu{l}\nu}\\
\tag{Jacobi's identity}=&{}{-C_{\mu{l}\nu}C_{\alpha\beta}{}^l}-C_{\mu{l}\nu}(\Gamma_{\alpha\beta}{}^l-\Gamma_{\beta\alpha}{}^l){}{-C_{\alpha{l}\nu}C_{\beta\mu}{}^l}-C_{\alpha{l}\nu}(\Gamma_{\beta\mu}{}^l-\Gamma_{\mu\beta}{}^l){}{-C_{\beta{l}\nu}C_{\mu\alpha}{}^l}\\
&-C_{\beta{l}\nu}(\Gamma_{\mu\alpha}{}^l-\Gamma_{\alpha\mu}{}^l)
-e_\alpha C_{\beta\mu\nu}-C_{\beta\mu}{}^{l}\Gamma_{\alpha{l}\nu}-e_\beta C_{\mu\alpha\nu}-C_{\mu\alpha}{}^{l}\Gamma_{\beta{l}\nu}-e_\mu C_{\alpha\beta\nu}-C_{\alpha\beta}{}^{l}\Gamma_{\mu{l}\nu}.
\end{align*}
The last RHS can be seen to correspond to the torsion terms in \eqref{fBian} by using the antisymmetries of $\Gamma_{\alpha\beta\nu},C_{\alpha\beta\nu}$ in the last two and first two indices respectively.

On the other hand, we have 
\begin{align*}
&{\bf \widetilde{D}}_\mu{\bf \widetilde{R}}(e_\alpha,e_\beta)
+{\bf \widetilde{D}}_\alpha{\bf \widetilde{R}}(e_\beta,e_\mu)
+{\bf \widetilde{D}}_\beta{\bf \widetilde{R}}(e_\mu,e_\alpha)\\
=&\,{\bf\widetilde{D}}_{e_\mu}({\bf \widetilde{D}}_{e_\alpha}{\bf \widetilde{D}}_{e_\beta}-{\bf \widetilde{D}}_{e_\beta}{\bf \widetilde{D}}_{e_\alpha}-{\bf \widetilde{D}}_{[e_\alpha,e_\beta]})
-({\bf\widetilde{D}}_{e_\mu}e_\alpha)^\nu({\bf \widetilde{D}}_{e_\nu}{\bf \widetilde{D}}_{e_\beta}-{\bf \widetilde{D}}_{e_\beta}{\bf \widetilde{D}}_{e_\nu}-{\bf \widetilde{D}}_{[e_\nu,e_\beta]})\\
&-({\bf\widetilde{D}}_{e_\mu}e_\beta)^\nu({\bf \widetilde{D}}_{e_\alpha}{\bf \widetilde{D}}_{e_\nu}-{\bf \widetilde{D}}_{e_\nu}{\bf \widetilde{D}}_{e_\alpha}-{\bf \widetilde{D}}_{[e_\alpha,e_\nu]})\\
&+{\bf\widetilde{D}}_{e_\alpha}({\bf \widetilde{D}}_{e_\beta}{\bf \widetilde{D}}_{e_\mu}-{\bf \widetilde{D}}_{e_\mu}{\bf \widetilde{D}}_{e_\beta}-{\bf \widetilde{D}}_{[e_\beta,e_\mu]})
-({\bf\widetilde{D}}_{e_\alpha}e_\beta)^\nu({\bf \widetilde{D}}_{e_\nu}{\bf \widetilde{D}}_{e_\mu}-{\bf \widetilde{D}}_{e_\mu}{\bf \widetilde{D}}_{e_\nu}-{\bf \widetilde{D}}_{[e_\nu,e_\mu]})\\
&-({\bf\widetilde{D}}_{e_\alpha}e_\mu)^\nu({\bf \widetilde{D}}_{e_\beta}{\bf \widetilde{D}}_{e_\nu}-{\bf \widetilde{D}}_{e_\nu}{\bf \widetilde{D}}_{e_\beta}-{\bf \widetilde{D}}_{[e_\beta,e_\nu]})\\
&+{\bf\widetilde{D}}_{e_\beta}({\bf \widetilde{D}}_{e_\mu}{\bf \widetilde{D}}_{e_\alpha}-{\bf \widetilde{D}}_{e_\alpha}{\bf \widetilde{D}}_{e_\mu}-{\bf \widetilde{D}}_{[e_\mu,e_\alpha]})
-({\bf\widetilde{D}}_{e_\beta}e_\mu)^\nu({\bf \widetilde{D}}_{e_\nu}{\bf \widetilde{D}}_{e_\alpha}-{\bf \widetilde{D}}_{e_\alpha}{\bf \widetilde{D}}_{e_\nu}-{\bf \widetilde{D}}_{[e_\nu,e_\alpha]})\\
&-({\bf\widetilde{D}}_{e_\beta}e_\alpha)^\nu({\bf \widetilde{D}}_{e_\mu}{\bf \widetilde{D}}_{e_\nu}-{\bf \widetilde{D}}_{e_\nu}{\bf \widetilde{D}}_{e_\mu}-{\bf \widetilde{D}}_{[e_\mu,e_\nu]})\\
=&\,[{\bf\widetilde{D}}_{e_\mu},{\bf\widetilde{D}}_{e_\alpha}]{\bf\widetilde{D}}_{e_\beta}+[{\bf\widetilde{D}}_{e_\beta},{\bf\widetilde{D}}_{e_\mu}]{\bf\widetilde{D}}_{e_\alpha}+[{\bf\widetilde{D}}_{e_\alpha},{\bf\widetilde{D}}_{e_\beta}]{\bf\widetilde{D}}_{e_\mu}-({\bf\widetilde{D}}_{e_\mu}e_\alpha-{\bf\widetilde{D}}_{e_\alpha}e_\mu)^\nu{\bf\widetilde{R}}(e_\nu,e_\beta)\\
&-({\bf\widetilde{D}}_{e_\alpha}e_\beta-{\bf\widetilde{D}}_{e_\beta}e_\alpha)^\nu{\bf\widetilde{R}}(e_\nu,e_\mu)
-({\bf\widetilde{D}}_{e_\beta}e_\mu-{\bf\widetilde{D}}_{e_\mu}e_\alpha)^\nu{\bf\widetilde{R}}(e_\nu,e_\alpha)\\
&-[{\bf\widetilde{D}}_{e_\mu},{\bf\widetilde{D}}_{[e_\alpha,e_\beta]}]
-[{\bf\widetilde{D}}_{e_\alpha},{\bf\widetilde{D}}_{[e_\beta,e_\mu]}]
-[{\bf\widetilde{D}}_{e_\beta},{\bf\widetilde{D}}_{[e_\mu,e_\alpha]}]\\
&-{\bf\widetilde{D}}_{[e_\alpha,e_\beta]}{\bf\widetilde{D}}_{e_\mu}
-{\bf\widetilde{D}}_{[e_\beta,e_\mu]}{\bf\widetilde{D}}_{e_\alpha}
-{\bf\widetilde{D}}_{[e_\mu,e_\alpha]}{\bf\widetilde{D}}_{e_\beta}\\
\tag{adding zero by Jacobi's identity}&+{\bf\widetilde{D}}_{[e_\mu,[e_\alpha,e_\beta]]}+{\bf\widetilde{D}}_{[e_\alpha,[e_\beta,e_\mu]]}+{\bf\widetilde{D}}_{[e_\beta,[e_\mu,e_\alpha]]}\\
=&\,{\bf\widetilde{R}}(e_\mu,e_\alpha)\widetilde{\bf D}_{e_\beta}
+{\bf\widetilde{R}}(e_\alpha,e_\beta)\widetilde{\bf D}_{e_\mu}
+{\bf\widetilde{R}}(e_\beta,e_\mu)\widetilde{\bf D}_{e_\alpha}
+([e_\alpha,e_\beta]-{\bf\widetilde{D}}_{e_\alpha}e_\beta+{\bf\widetilde{D}}_{e_\beta}e_\alpha)^\nu{\bf\widetilde{R}}(e_\nu,e_\mu)\\
&+([e_\beta,e_\mu]-{\bf\widetilde{D}}_{e_\beta}e_\mu+{\bf\widetilde{D}}_{e_\mu}e_\beta)^\nu{\bf\widetilde{R}}(e_\nu,e_\alpha)
+([e_\mu,e_\alpha]-{\bf\widetilde{D}}_{e_\mu}e_\alpha+{\bf\widetilde{D}}_{e_\alpha}e_\mu)^\nu{\bf\widetilde{R}}(e_\nu,e_\beta)
\end{align*}
The second Bianchi identity follows by applying the preceding expression to $e_\gamma$, taking the inner product with $e_\delta$ and utilising the idenitty $[e_\alpha,e_\beta]=C_{\alpha\beta}{}^le_l+{\bf\widetilde{D}}_{e_\alpha}e_\beta-{\bf\widetilde{D}}_{e_\beta}e_\alpha$.

Finally, for the Gauss and Codazzi equations \eqref{Gausstilde}-\eqref{Codtilde}, we repeat the steps in the proof of Lemma \ref{lem:GaussCod}, making use of the formula $[e_j,e_b]=C_{jb}{}^le_l+{\bf\widetilde{D}}_{e_j}e_b-{\bf\widetilde{D}}_{e_b}e_j$, without identifying ${\bf \widetilde{R}}_{0ijb},{\bf \widetilde{R}}_{jb0i}$ (which uses the torsion free property of the connection). This completes the proof of the lemma.
\end{proof}
\subsection{Modified curvature and propagation equations for vanishing quantities}

{}{An essential step in proving the vanishing of $C_{ijb},\widetilde{\bf R}_{\beta\mu}$ is the derivation of propagation equations for them, using the reduced equations \eqref{e0Kijnew}-\eqref{e0Gammaijbnew} and the Bianchi identities for the curvature of $\widetilde{\bf D}$ in Lemma \ref{lem:Dtilde}. If these equations are suitable for an energy argument, then we can infer the vanishing of the relevant variables from their vanishing on the initial hypersurface.} 

However, for the particular curvature of $\widetilde{\bf D}$ we have defined,  this system would fail to be hyperbolic, hence, obstructing us from deriving energy estimates. {}{Indeed, this can be seen by examining the system of evolution equations in Lemma \ref{lem:systRic}, which we derive below for the modified curvature \eqref{curvhat}. The first order system \eqref{e0Cijb}-\eqref{RijAeq} is in fact symmetric hyperbolic, but if we were to replace $\widehat{\bf R}$ by $\widetilde{\bf R}$ this would fail to be the case, due to the additional first order $C_{ijb}$ terms with no particular structure.}

For this purpose, we consider the modified curvature:
\begin{align}\label{curvhat}
{\bf \widehat{R}}_{\alpha\beta\mu}{}^\nu e_\nu:=({\bf \widetilde{D}}_{e_\alpha}{\bf \widetilde{D}}_{e_\beta}-{\bf \widetilde{D}}_{e_\beta}{\bf \widetilde{D}}_{e_\alpha}-{\bf \widetilde{D}}_{\widetilde{\bf D}_{e_\alpha}e_\beta-\widetilde{\bf D}_{e_\beta}e_\alpha})e_\mu=(\widetilde{\bf R}_{\alpha\beta\mu}{}^\nu+C_{\alpha\beta}{}^\lambda\Gamma_{\lambda\mu}{}^\nu) e_\nu
\end{align}
Note that ${\bf \widehat{R}}_{\alpha\beta\mu\nu}$ is not tensorial with respect to its third index $\mu$. We also define ${\bf \widehat{R}}_{\beta\mu}={\bf \widehat{R}}_{\alpha \beta\mu}{}^\alpha,\widehat{\bf R}={\bf\widehat{R}}_\mu{}^\mu$ and similarly for the modified curvatures $\widehat{R}_{aijb},\widehat{R}_{ij},\widehat{R}$ of $\widetilde{D}$. Then we have the following identities, which are immediate consequences of Lemma \ref{lem:Dtilde} and \eqref{curvhat}:
\begin{lemma}\label{lem:Dhat}
The curvatures ${\bf \widehat{R}}_{\alpha\beta\mu\nu},\widehat{R}_{aijb}$ satisfy the identites:
\begin{align}
\label{hatantisRiem}
{\bf \widehat{R}}_{\alpha\beta\mu\nu}=&-{\bf \widehat{R}}_{\beta\alpha\mu\nu}=-{\bf\widehat{R}}_{\alpha\beta\nu\mu},\qquad \widehat{R}_{aijb}=-\widehat{R}_{iajb}=-\widehat{R}_{aibj}\\
\label{hatfBian}0=&\,{\bf \widehat{R}}_{\alpha\beta\mu\nu}+{\bf \widehat{R}}_{\beta\mu\alpha\nu}+{\bf \widehat{R}}_{\mu\alpha\beta\nu}
+{\bf \widetilde{D}}_\mu C_{\alpha\beta\nu}+{\bf \widetilde{D}}_\alpha C_{\beta\mu\nu}+{\bf \widetilde{D}}_\beta C_{\mu\alpha\nu}\\
\notag&{}{+C_{\alpha\beta}{}^l C_{l\mu\nu}+C_{\mu\alpha}{}^l C_{l\beta\nu}+C_{\beta\mu}{}^l C_{l\alpha\nu}-C_{\alpha\beta}{}^\lambda\Gamma_{\lambda\mu\nu}-C_{\beta\mu}{}^\lambda\Gamma_{\lambda\alpha\nu}-C_{\mu\alpha}{}^\lambda\Gamma_{\lambda\beta\nu}}\\
\label{hataRic}{\bf\widehat{R}}_{\alpha\beta}-{\bf\widehat{R}}_{\beta\alpha}
=&-{\bf \widetilde{D}}^\mu C_{\alpha\beta\mu}-{\bf \widetilde{D}}_\alpha C_{\beta\mu}{}^\mu-{\bf \widetilde{D}}_\beta C_{\mu\alpha}{}^\mu\\
\notag&{}{-C_{\alpha\beta}{}^l C_{l\mu}{}^\mu-C_{\mu\alpha}{}^l C_{l\beta}{}^\mu-C_{\beta\mu}{}^l C_{l\alpha}{}^\mu+C_{\beta\mu}{}^\lambda\Gamma_{\lambda\alpha}{}^\mu+C_{\mu\alpha}{}^\lambda\Gamma_{\lambda\beta}{}^\mu}\\
\label{hatsBian}0=&\,{\bf \widetilde{D}}_\mu{\bf \widehat{R}}_{\alpha\beta\gamma\delta}
+{\bf \widetilde{D}}_\alpha{\bf \widehat{R}}_{\beta\mu\gamma\delta}
+{\bf \widetilde{D}}_\beta{\bf \widehat{R}}_{\mu\alpha\gamma\delta}
-C_{\mu\alpha l}({\bf \widehat{R}}^l{}_{\beta\gamma\delta}-C^l{}_{\beta\nu}\Gamma^\nu{}_{\gamma\delta})\\
&\notag-C_{\alpha\beta l}({\bf \widehat{R}}^l{}_{\mu\gamma\delta}-C^l{}_{\mu\nu}\Gamma^\nu{}_{\gamma\delta})-C_{\beta\mu l}({\bf \widehat{R}}^l{}_{\alpha\gamma\delta}-C^l{}_{\alpha\nu}\Gamma^\nu{}_{\gamma\delta})\\
\notag&+\bigg[{\bf \widehat{R}}_{\alpha\beta\mu\nu}+{\bf \widehat{R}}_{\beta\mu\alpha\nu}+{\bf \widehat{R}}_{\mu\alpha\beta\nu}{}{+
C_{\alpha\beta}{}^l C_{l\mu\nu}+C_{\mu\alpha}{}^l C_{l\beta\nu}+C_{\beta\mu}{}^l C_{l\alpha\nu}}\\
&{}{-C_{\alpha\beta}{}^\lambda\Gamma_{\lambda\mu\nu}-C_{\beta\mu}{}^\lambda\Gamma_{\lambda\alpha\nu}-C_{\mu\alpha}{}^\lambda\Gamma_{\lambda\beta\nu}}
\bigg]\Gamma_{\nu\gamma\delta}\\
\notag&-C_{\alpha\beta}{}^\nu{\bf \widetilde{D}}_\mu\Gamma_{\nu\gamma\delta}
-C_{\beta\mu}{}^\nu{\bf \widetilde{D}}_\alpha\Gamma_{\nu\gamma\delta}
-C_{\mu\alpha}{}^\nu{\bf \widetilde{D}}_\beta\Gamma_{\nu\gamma\delta}\\
\label{Gausshat}
{\bf \widehat{R}}_{aijb}=&\,\widehat{R}_{aijb}+K_{ab}K_{ij}-K_{aj}K_{ib},\\
\label{Codhat}{\bf \widehat{R}}_{jb0i}=&\,\widetilde{D}_jK_{bi}-\widetilde{D}_bK_{ji},\\
\label{hatRb0}{\bf \widehat{R}}_{b0}=&\,\widetilde{D}^iK_{bi}-\widetilde{D}_b\mathrm{tr}K\\
\label{hatRaijb}\widehat{R}_{aijb}=&\,e_a\Gamma_{ijb}-e_i\Gamma_{ajb}-\Gamma_{ab}{}^c\Gamma_{ijc}+\Gamma_{ib}{}^c\Gamma_{ajc}-\Gamma_{ai}{}^c\Gamma_{cjb}+\Gamma_{ia}{}^c\Gamma_{cjb}
\end{align}
where everything is interpreted tensorially, e.g., ${\bf \widetilde{D}}_\mu\Gamma_{\nu\gamma\delta}:=e_\mu\Gamma_{\nu\gamma\delta}-\Gamma_{\mu\nu}{}^\lambda\Gamma_{\lambda\gamma\delta}-\Gamma_{\mu\gamma}{}^\lambda\Gamma_{\nu\lambda\delta}-\Gamma_{\mu\delta}{}^\lambda\Gamma_{\nu\gamma\lambda}$.
\end{lemma}
\begin{proof}
The antisymmetries \eqref{hatantisRiem} follow from the definition \eqref{curvhat}, the antisymmetries \eqref{curvtilde}, \eqref{antisRiem} of ${\bf\widetilde{R}}_{\alpha\beta\mu\nu}$ and that of $C_{\alpha\beta\mu}$ in $(\alpha;\beta)$. Also, plugging \eqref{curvhat} into \eqref{fBian} gives \eqref{hatfBian}, while contracting \eqref{hatfBian} with respect to $(\mu;\nu)$ gives \eqref{hataRic}. Moreover, \eqref{Gausshat}-\eqref{Codhat} follow from \eqref{Gausstilde}-\eqref{Codtilde} by plugging in the definition \eqref{curvhat} and recalling that $C_{ai0}=0$, see \eqref{Cijb}. Contracting \eqref{Codhat} also gives \eqref{hatRb0}. The computation of the curvature formula \eqref{hatRaijb} is straightforward, using the definition of $\widehat{R}_{aijb}$, analogous to \eqref{curvhat}, cf. the proof of Lemma \ref{lem:GaussCod}.
]

For the less obvious Bianchi-type identity \eqref{hatsBian}, we plug \eqref{curvhat} into \eqref{sBian} and treat all the terms tensorially. Although $\widehat{\bf R}_{\alpha\beta\gamma\delta}$ is not a tensor in $\gamma$, its difference from $C_{\alpha\beta}{}^\nu\Gamma_{\nu\gamma\delta}$ is. Therefore, we deduce
\begin{align*}
0=&\,{\bf \widetilde{D}}_\mu({\bf \widehat{R}}_{\alpha\beta\gamma\delta}-C_{\alpha\beta}{}^\nu\Gamma_{\nu\gamma\delta})
+{\bf \widetilde{D}}_\alpha({\bf \widehat{R}}_{\beta\mu\gamma\delta}-C_{\beta\mu}{}^\nu\Gamma_{\nu\gamma\delta})
+{\bf \widetilde{D}}_\beta({\bf \widehat{R}}_{\mu\alpha\gamma\delta}-C_{\mu\alpha}{}^\nu\Gamma_{\nu\gamma\delta})\\
&-C_{\mu\alpha l}({\bf \widehat{R}}^l{}_{\beta\gamma\delta}-C^l{}_{\beta\nu}\Gamma^\nu{}_{\gamma\delta})-C_{\alpha\beta l}({\bf \widehat{R}}^l{}_{\mu\gamma\delta}-C^l{}_{\mu\nu}\Gamma^\nu{}_{\gamma\delta})-C_{\beta\mu l}({\bf \widehat{R}}^l{}_{\alpha\gamma\delta}-C^l{}_{\alpha\nu}\Gamma^\nu{}_{\gamma\delta})\\
=&\,{\bf \widetilde{D}}_\mu{\bf \widehat{R}}_{\alpha\beta\gamma\delta}
+{\bf \widetilde{D}}_\alpha{\bf \widehat{R}}_{\beta\mu\gamma\delta}
+{\bf \widetilde{D}}_\beta{\bf \widehat{R}}_{\mu\alpha\gamma\delta}\\
&-C_{\mu\alpha l}({\bf \widehat{R}}^l{}_{\beta\gamma\delta}-C^l{}_{\beta\nu}\Gamma^\nu{}_{\gamma\delta})-C_{\alpha\beta l}({\bf \widehat{R}}^l{}_{\mu\gamma\delta}-C^l{}_{\mu\nu}\Gamma^\nu{}_{\gamma\delta})-C_{\beta\mu l}({\bf \widehat{R}}^l{}_{\alpha\gamma\delta}-C^l{}_{\alpha\nu}\Gamma^\nu{}_{\gamma\delta})\\
&{}{-({\bf \widetilde{D}}_\mu C_{\alpha\beta}{}^\nu+
{\bf \widetilde{D}}_\alpha C_{\beta\mu}{}^\nu
+{\bf \widetilde{D}}_\beta C_{\mu\alpha}{}^\nu)\Gamma_{\nu\gamma\delta}}
{}{-C_{\alpha\beta}{}^\nu {\bf \widetilde{D}}_\mu\Gamma_{\nu\gamma\delta}
-C_{\beta\mu}{}^\nu{\bf \widetilde{D}}_\alpha \Gamma_{\nu\gamma\delta}
-C_{\mu\alpha}{}^\nu{\bf \widetilde{D}}_\beta\Gamma_{\nu\gamma\delta}}
\end{align*}
On the other hand, we employ the first Bianchi identity \eqref{hatfBian} to write
\begin{align*}
&-{\bf \widetilde{D}}_\mu C_{\alpha\beta}{}^\nu
-{\bf \widetilde{D}}_\alpha C_{\beta\mu}{}^\nu
-{\bf \widetilde{D}}_\beta C_{\mu\alpha}{}^\nu\\
=&\,{\bf \widehat{R}}_{\alpha\beta\mu\nu}+{\bf \widehat{R}}_{\beta\mu\alpha\nu}+{\bf \widehat{R}}_{\mu\alpha\beta\nu}
{}{+C_{\alpha\beta}{}^l C_{l\mu\nu}+C_{\mu\alpha}{}^l C_{l\beta\nu}+C_{\beta\mu}{}^l C_{l\alpha\nu}}\\
&{}{-C_{\alpha\beta}{}^\lambda\Gamma_{\lambda\mu\nu}-C_{\beta\mu}{}^\lambda\Gamma_{\lambda\alpha\nu}-C_{\mu\alpha}{}^\lambda\Gamma_{\lambda\beta\nu}}
\end{align*}
which completes the proof of the lemma.
\end{proof}
\begin{remark}
It is important that \eqref{hatsBian} does not contain any spatial derivatives of $C_{ijb}$, which could lead to a non-symmetric system for the vanishing variables, cf. Lemma \ref{lem:systRic}. We were able to replace such terms by using the first Bianchi identity \eqref{hatfBian}. In turn, we must express the cyclic curvature sum in \eqref{hatsBian} solely by Ricci terms. 
\end{remark}
\begin{lemma}\label{lem:cyclRiem}
The cyclic sum ${\bf \widehat{R}}_{(\alpha\beta\mu)\nu}:={\bf \widehat{R}}_{\alpha\beta\mu\nu}+{\bf \widehat{R}}_{\beta\mu\alpha\nu}+{\bf \widehat{R}}_{\mu\alpha\beta\nu}$ satisfies:
\begin{align}
\label{cyclRiem}
\begin{split}
{\bf \widehat{R}}_{(abi)j}=({\bf\widehat{R}}_{ia}-{\bf\widehat{R}}_{ai}){\bf g}_{bj}+ ({\bf\widehat{R}}_{bi}-{\bf\widehat{R}}_{ib}){\bf g}_{aj}+({\bf\widehat{R}}_{ab}-{\bf\widehat{R}}_{ba}){\bf g}_{ij},\\
{\bf \widehat{R}}_{(0bi)0}={\bf \widehat{R}}_{(abi)0}=0,\qquad{\bf \widehat{R}}_{(0bi)j}=-\delta_{ij}{\bf\widehat{R}}_{b0}+\delta_{bj}{\bf\widehat{R}}_{i0}
\end{split}
\end{align}
\end{lemma}
\begin{proof}
For the first identity, we notice that if either of $a,b,i$ coincide, both sides are trivially zero. In the case where $a,b,i$ are all distinct, $j$ must coincide with one of them (since $\Sigma_t$ is $3$-dimensional), say $j=a$. Then we have 
\begin{align}\label{cyclspatial}
{\bf\widehat{R}}_{bi}-{\bf\widehat{R}}_{ib}=&\,{\bf\widehat{R}}_{\lambda bi}{}^\lambda-{\bf\widehat{R}}_{\lambda ib}{}^\lambda\overset{\eqref{hatantisRiem}}{=}{\bf\widehat{R}}_{abia}+{\bf\widehat{R}}_{iaba}+{\bf\widehat{R}}_{0b0i}-{\bf\widehat{R}}_{0i0b}={\bf\widehat{R}}_{(abi)a}+{\bf\widehat{R}}_{0b0i}-{\bf\widehat{R}}_{0i0b}
\end{align}
On the other hand, using the symmetry of $K_{bi}$ it holds
\begin{align}\label{R0b0i=R0i0b}
{\bf\widehat{R}}_{0b0i}=&\,{\bf g}(({\bf \widetilde{D}}_{e_0}{\bf \widetilde{D}}_{e_b}-{\bf \widetilde{D}}_{e_b}{\bf \widetilde{D}}_{e_0}-{\bf \widetilde{D}}_{\widetilde{\bf D}_{e_0}e_b-\widetilde{\bf D}_{e_b}e_0})e_0,e_i)=e_0K_{bi}+K_b{}^lK_{li}={\bf\widehat{R}}_{0i0b},
\end{align}
Combining \eqref{cyclspatial}-\eqref{R0b0i=R0i0b} yeilds the first identity in \eqref{cyclRiem}. Also, \eqref{R0b0i=R0i0b} implies the first part of the second identity in \eqref{cyclRiem} regarding ${\bf \widehat{R}}_{(0bi)0}=0$.

Next, we employ \eqref{Codhat} to infer:
\begin{align*}
{\bf \widehat{R}}_{(abi)0}=&-{\bf \widehat{R}}_{ab0i}-{\bf \widehat{R}}_{bi0a}-{\bf \widehat{R}}_{ia0b}\\
=&-\widetilde{D}_aK_{bi}+\widetilde{D}_bK_{ai}-\widetilde{D}_bK_{ia}+\widetilde{D}_iK_{ba}-\widetilde{D}_iK_{ab}+\widetilde{D}_aK_{ib}\\
=&\,0
\end{align*}
{}{To prove the last identity in \eqref{cyclRiem} we utilise the reduced equation \eqref{e0Gammaijbnew}, which we rewrite in a more covariant way using \eqref{hatRb0}:
\begin{align}\label{eq.Gammanewcov}
e_0\Gamma_{ijb}+K_i{}^c\Gamma_{cjb}=\widetilde{D}_jK_{bi}-\widetilde{D}_bK_{ji}+\delta_{ib}\widehat{\bf R}_{j0}
-\delta_{ij}\widehat{\bf R}_{b0}
\end{align}
Appealing to the symmetry of $K$ once more, we compute:}  
\begin{align*}
{\bf \widehat{R}}_{(0bi)j}=&\,{\bf \widehat{R}}_{0bij}+{\bf \widehat{R}}_{bi0j}-{\bf \widehat{R}}_{0ibj}\\
=&\,{\bf g}(({\bf \widetilde{D}}_{e_0}{\bf \widetilde{D}}_{e_b}-{\bf \widetilde{D}}_{e_b}{\bf \widetilde{D}}_{e_0}-{\bf \widetilde{D}}_{\widetilde{\bf D}_{e_0}e_b-\widetilde{\bf D}_{e_b}e_0})e_i,e_j)\\
&\tag{by \eqref{Codhat}}+\widetilde{D}_bK_{ij}-\widetilde{D}_iK_{bj}
-{\bf g}(({\bf \widetilde{D}}_{e_0}{\bf \widetilde{D}}_{e_i}-{\bf \widetilde{D}}_{e_i}{\bf \widetilde{D}}_{e_0}-{\bf \widetilde{D}}_{\widetilde{\bf D}_{e_0}e_i-\widetilde{\bf D}_{e_i}e_0})e_b,e_j)\\
=&\,e_0\Gamma_{bij}+K_b{}^c\Gamma_{cij}-e_0\Gamma_{ibj}-K_i{}^c\Gamma_{cbj}
+\widetilde{D}_bK_{ij}-\widetilde{D}_iK_{bj}\\
\tag{by \eqref{eq.Gammanewcov}}=&\,\widetilde{D}_iK_{jb}-\widetilde{D}_bK_{ji}-\delta_{ij}{\bf\widehat{R}}_{b0}+\delta_{bj}{\bf\widehat{R}}_{i0}
+\widetilde{D}_bK_{ij}-\widetilde{D}_iK_{bj}\\
=&-\delta_{ij}{\bf\widehat{R}}_{b0}+\delta_{bj}{\bf\widehat{R}}_{i0},
\end{align*}
as asserted.
\end{proof}

Recall that we symmetrized the RHS of \eqref{e0Kij}, such that the symmetry of $K_{ij}$ is automatically propagated off of the initial hypersurface. Consequently, we must treat the symmetrized and antisymmetrized Ricci tensors as different variables:
\begin{align}\label{hatASRic}
{\bf\widehat{R}}^{{}{(S)}}_{ij}=\frac{1}{2}({\bf\widehat{R}}_{ij}+{\bf\widehat{R}}_{ji})={\bf\widehat{R}}^{{}{(S)}}_{ji},\qquad {\bf\widehat{R}}^{{}{(A)}}_{ij}=\frac{1}{2}({\bf\widehat{R}}_{ij}-{\bf\widehat{R}}_{ji})=-{\bf\widehat{R}}^{{}{(A)}}_{ji}.
\end{align}
With the above lemmas at our disposal, we derive the following propagation equations for the variables that should vanish.
\begin{lemma}\label{lem:systRic}
The variables ${\bf \widehat{R}}_{\beta\mu},C_{ijb}$ satisfy the following system of equations:
\begin{align}
\label{e0Cijb}e_0C_{ijb}=&\,{}{K_b{}^lC_{ijl}}-K_i{}^lC_{ljb}-K_j{}^lC_{ilb}-\delta_{ib}{\bf \widehat{R}}_{j0}+\delta_{jb}{\bf \widehat{R}}_{i0},\\
\label{e0R0b}e_0{\bf \widehat{R}}_{i0}=&\,e_i{\bf \widehat{R}}_{00}
+e^a{\bf \widehat{\bf R}}_{ia}^{{}{(A)}}
-\Gamma^a{}_i{}^{\beta}{\bf \widehat{\bf R}}_{\beta a}-\Gamma_a{}^{a\beta}{\bf \widehat{\bf R}}_{i\beta}
{}{-L_i(C,\widehat{\bf R})},\qquad {\bf \widehat{R}}_{0i}=-{\bf \widehat{R}}_{i0},\\
\label{R00eq}e_0{\bf \widehat{R}}_{00}=&\,e^a{\bf \widehat{R}}_{a0}
-\Gamma_a{}^{ab}{\bf \widehat{R}}_{b0}{}{+L_0(C,\widehat{\bf R})},\\
\label{RijAeq}e_0{\bf \widehat{R}}_{ij}^{{}{(A)}}=&\,\frac{1}{2}(e_j{\bf\widehat{R}}_{i0}-e_i{\bf\widehat{R}}_{j0}){}{-K_i{}^l{\bf\widehat{R}}_{lj}^{{}{(A)}}}-K_j{}^l{\bf\widehat{R}}_{il}^{{}{(A)}}
-\frac{1}{2}K^{bl}{\bf\widehat{R}}_{(ijl)b}{}{+\frac{1}{2}K^{bl}\widehat{\bf R}_{(ijb)l}+{}{M_{ij}(C,\widehat{\bf R})}},\\
\label{Rijeq}{\bf \widehat{R}}_{ij}^{{}{(S)}}=&-\delta_{ij}{\bf \widehat{R}}_{00},
\end{align}
where 
\begin{align}
\notag{}{2 L_\mu(C,\widehat{\bf R})}=&-C_{\mu\alpha l}({\bf \widehat{R}}^l{}_{\beta\gamma\delta}-C^l{}_{\beta\nu}\Gamma^\nu{}_{\gamma\delta})-C_{\alpha\beta l}({\bf \widehat{R}}^l{}_{\mu\gamma\delta}-C^l{}_{\mu\nu}\Gamma^\nu{}_{\gamma\delta})-C_{\beta\mu l}({\bf \widehat{R}}^l{}_{\alpha\gamma\delta}-C^l{}_{\alpha\nu}\Gamma^\nu{}_{\gamma\delta})\\
\label{Lmu}&+\bigg[{\bf \widehat{R}}_{\alpha\beta\mu\nu}+{\bf \widehat{R}}_{\beta\mu\alpha\nu}+{\bf \widehat{R}}_{\mu\alpha\beta\nu}{}{+
C_{\alpha\beta}{}^l C_{l\mu\nu}+C_{\mu\alpha}{}^l C_{l\beta\nu}+C_{\beta\mu}{}^l C_{l\alpha\nu}}{}{-C_{\alpha\beta}{}^\lambda\Gamma_{\lambda\mu\nu}}\\
&\notag{}{-C_{\beta\mu}{}^\lambda\Gamma_{\lambda\alpha\nu}-C_{\mu\alpha}{}^\lambda\Gamma_{\lambda\beta\nu}}
\bigg]\Gamma_{\nu\gamma\delta}-C_{\alpha\beta}{}^\nu{\bf \widetilde{D}}_\mu\Gamma_{\nu\gamma\delta}
-C_{\beta\mu}{}^\nu{\bf \widetilde{D}}_\alpha\Gamma_{\nu\gamma\delta}
-C_{\mu\alpha}{}^\nu{\bf \widetilde{D}}_\beta\Gamma_{\nu\gamma\delta},\\
\notag {}{M_{ij}(C,\widehat{\bf R})}=&-\frac{1}{2}e_0L_{ij}(C)-\frac{1}{2}K_i{}^lL_{lj}(C)-\frac{1}{2}K_j{}^lL_{il}(C)-\frac{1}{2}K^{bl}L_{ijlb}(C)+\frac{1}{2}K^{bl}L_{ijbl}(C)\\
\label{M(C)}&-\frac{1}{2}\bigg[C_{ijl}e^bK_b{}^l+C_{jbl}e_iK^{bl}+C_{bil}e_jK^{bl}
-C_{ljb}e^bK_i{}^l-C_{ilb}e^bK_j{}^l
-C_{lb}{}^be_iK_j{}^l\\
\notag&-C_{jlb}e_iK^{bl}
-C_{lib}e_jK^{bl}-C_{bl}{}^be_jK_i{}^l\bigg],\\
\label{Lij}
{}{L_{ij}(C)}=&\,\widetilde{D}^b C_{ijb}+\widetilde{D}_i C_{jb}{}^b+\widetilde{D}_j C_{bi}{}^b-e^b C_{ijb}-e_i C_{jb}{}^b-e_j C_{bi}{}^b\\
\notag&+C_{ij}{}^l C_{lb}{}^b+C_{b i}{}^l C_{lj}{}^b+C_{jb}{}^l C_{li}{}^b-C_{jb}{}^l\Gamma_{l i}{}^b-C_{b i}{}^l\Gamma_{l j}{}^b,\\
\label{Lijbl}
{}{L_{ijbl}(C)}=&\,\widetilde{D}_bC_{ijl}+\widetilde{D}_iC_{jbl}+\widetilde{D}_jC_{bil}-e_bC_{ijl}-e_iC_{jbl}-e_jC_{bil}\\
\notag&-C_{ij}{}^d C_{dbl}-C_{bi}{}^d C_{djl}-C_{jb}{}^d C_{dil}+C_{ij}{}^d\Gamma_{dbl}+C_{jb}{}^d\Gamma_{dil}+C_{bi}{}^d\Gamma_{djl}.
\end{align}
\end{lemma}
\begin{remark}\label{rem:Ricsym}
{}{The system \eqref{e0Cijb}-\eqref{RijAeq} constitutes a (linear, homogeneous) first order symmetric hyperbolic system for the variables $C_{ijb},\widehat{\bf R}_{i0},\widehat{\bf R}_{00},\widehat{\bf R}_{ij}^{{}{(A)}}$. Indeed, we notice that $L_\mu(C,\widehat{R}),M_{ij}(C,\widehat{\bf R}),\widehat{\bf R}_{(ijb)l}$ can be viewed, by virtue of Lemma \ref{lem:cyclRiem} and \eqref{Cijb}, \eqref{e0Cijb}, as linear expressions in the unknowns, with coefficients depending on the solution $K_{ij},\Gamma_{ijb},f_i{}^j,f^b{}_j$ to the reduced equations \eqref{e0fij}, \eqref{e0fijinv}, \eqref{e0Kijnew}, \eqref{e0Gammaijbnew} and their first derivatives.}
\end{remark}
\begin{proof}
We compute \eqref{e0Cijb} by directly differentiating \eqref{Cijb} and using the commutation formula \eqref{comm}, the evolution equations \eqref{e0fij}, \eqref{e0fijinv}, \eqref{eq.Gammanewcov}, Lemma \ref{lem:symm}, along with \eqref{hatRb0}: 
\begin{align}\label{e0Cijb2}
\notag e_0C_{ijb}=&\,e_0(f^b{}_le_if_j{}^l-f^b{}_le_jf_i{}^l-\Gamma_{ijb}+\Gamma_{jib})\\
\notag=&\,K_c{}^bf^c{}_le_if_j{}^l-f^b{}_le_i(K_j{}^cf_c{}^l)-f^b{}_lK_i{}^ce_cf_j{}^l-K_c{}^bf^c{}_le_jf_i{}^l+f^b{}_le_j(K_i{}^cf_c{}^l)+f^b{}_lK_j{}^ce_cf_i{}^l\\
\notag&-\bigg[-K_i{}^l\Gamma_{ljb}+\widetilde{D}_jK_{bi}-\widetilde{D}_bK_{ji}
+\delta_{ib}\widehat{\bf R}_{j0}-\delta_{ij}\widehat{\bf R}_{b0}\bigg]\\
\notag&+\bigg[-K_j{}^l\Gamma_{lib}+\widetilde{D}_iK_{bj}-\widetilde{D}_bK_{ij}
+\delta_{jb}\widehat{\bf R}_{i0}-\delta_{ij}\widehat{\bf R}_{b0}\bigg]\\
=&\,{}{K_b{}^lC_{ijl}}-K_i{}^lC_{ljb}-K_j{}^lC_{ilb}-\delta_{ib}{\bf \widehat{R}}_{j0}+\delta_{jb}{\bf\widehat{R}}_{i0}
\end{align}
Moreover, a direct computation shows that
\begin{align}
\notag{\bf \widehat{R}}_{0b}=-{\bf \widehat{R}}_{0ib}{}^i=&-{\bf g}(({\bf \widetilde{D}}_{e_0}{\bf \widetilde{D}}_{e_i}-{\bf \widetilde{D}}_{e_i}{\bf \widetilde{D}}_{e_0}-{\bf \widetilde{D}}_{\widetilde{\bf D}_{e_0}e_i-\widetilde{\bf D}_{e_i}e_0})e_b,e^i)\\
\label{Rb0R0bhat}=&-e_0\Gamma_{ib}{}^i-K_i{}^c\Gamma_{cb}{}^i\overset{\eqref{eq.Gammanewcov},\eqref{hatRb0}}{=}-{\bf \widehat{R}}_{b0}
\end{align}
Also, contracting \eqref{Gausshat} in $(a;b)$ and $(i;j)$ we obtain
\begin{align}
\label{contrGausshat}{\bf \widehat{R}}+2{\bf \widehat{R}}_{00}=&\,\widehat{R}-|K|^2+(\text{tr}K)^2,
\end{align}
while 
\begin{align}
\notag{\bf \widehat{R}}_{0i0j}=&-{\bf \widehat{R}}_{0ij0}={\bf \widehat{R}}_{ij}-{\bf \widehat{R}}_{bij}{}^b\overset{\eqref{Gausshat}}{=}
{\bf \widehat{R}}_{ij}-\widehat{R}_{ij}-\text{tr}KK_{ij}+{K_i}^bK_{jb}\\
\tag{by \eqref{R0b0i=R0i0b}}{\bf \widehat{R}}_{0i0j}=&\,e_0K_{ij}+K_i{}^bK_{bj}\\
\label{R0i0jhat}&\Rightarrow e_0K_{ij}+\text{tr}K K_{ij}=-\widehat{R}_{ij}+{\bf\widehat{R}}_{ij}
\end{align} 
Contracting \eqref{hatRaijb} and using the antisymmetry of $\Gamma_{ijb}$ (see Lemma \ref{lem:symm}), the spatial Ricci tensor in the preceding RHS expands to 
\begin{align}\label{Rijhat}
-\widehat{R}_{ij}:=-\widehat{R}_{bij}{}^b=e_i\Gamma^b{}_{jb}-e^b\Gamma_{ijb}
{}{+\Gamma^b{}_i{}^c\Gamma_{cjb}+\Gamma^b{}_b{}^c\Gamma_{ijc}}
\end{align}
By the symmetry of $K_{ij}$ we also have 
\begin{align}\label{e0Kijhat}
e_0K_{ij}+\text{tr}K K_{ij}=-\widehat{R}_{ij}^{{}{(S)}}+{\bf\widehat{R}}_{ij}^{{}{(S)}}.
\end{align} 
Due to \eqref{Rijhat}, we find that the reduced equation \eqref{e0Kijnew} corresponds to {}{(cf. \eqref{spatialR} and Remark \ref{rem:addterm})}
\begin{align}\label{e0KijRic2}
e_0K_{ij}+\text{tr}KK_{ij}=&-\widehat{R}_{ij}^{{}{(S)}}+\frac{1}{2}\delta_{ij}[\widehat{R}-|K|^2+(\text{tr}K)^2]
\end{align}
Combining \eqref{contrGausshat}-\eqref{e0KijRic2}we deduce the identities:
\begin{align}\label{Ricid}
\frac{1}{2}({\bf \widehat{R}}_{ij}+{\bf \widehat{R}}_{ji})=&\,\frac{1}{2}\delta_{ij}[{\bf \widehat{R}}+2{\bf \widehat{R}}_{00}]
\end{align}
Contracting indices in \eqref{Ricid} gives 
\begin{align}\label{Rhat}
{\bf \widehat{R}}+{\bf \widehat{R}}_{00}=&\,\frac{3}{2}[{\bf \widehat{R}}+2{\bf \widehat{R}}_{00}]\qquad\Rightarrow\qquad {\bf \widehat{R}}=-4{\bf \widehat{R}}_{00}\qquad
\Rightarrow\qquad\frac{1}{2}({\bf \widehat{R}}_{ij}+{\bf \widehat{R}}_{ji})=-\delta_{ij}{\bf \widehat{R}}_{00},
\end{align}
which confirms \eqref{Rijeq}.

Next, we contract the second Bianchi identity \eqref{hatsBian}
in the indices $(\alpha;\delta)$ and $(\beta;\gamma)$ to obtain:
\begin{align}\label{contrsecondBianchi}
{\bf \widetilde{D}}^\alpha{\bf \widehat{R}}_{\mu\alpha}=&\,\frac{1}{2}{\bf \widetilde{D}}_\mu{\bf \widehat{R}}+L_\mu(C,{\bf\widehat{R}}),
\end{align}
where $L_\mu(C,{\bf\widehat{R}})$ is given by \eqref{Lmu}.
Hence, for $\mu=i=1,2,3$, we deduce the equation 
\begin{align*}
e_0{\bf \widehat{R}}_{i0}\overset{\eqref{Ddef}}{=}&\,{\bf \widetilde{D}}_0{\bf \widehat{R}}_{i0}=-\frac{1}{2}e_i{\bf \widehat{R}}+{\bf \widetilde{D}}^a{\bf \widehat{\bf R}}_{ia}{}{-L_i(C,\widehat{\bf R})}\\
\overset{\eqref{Rhat}}{=}&\,2e_i{\bf \widehat{R}}_{00}+e^a{\bf \widehat{\bf R}}_{ia}^{{}{(S)}}+e^a{\bf \widehat{\bf R}}_{ia}^{{}{(A)}}
-\Gamma^a{}_i{}^{\beta}{\bf \widehat{\bf R}}_{\beta a}-\Gamma_a{}^{a\beta}{\bf \widehat{\bf R}}_{i\beta}
{}{-L_i(C,\widehat{\bf R})}\\
\overset{\eqref{Rijeq}}{=}&\,e_i{\bf \widehat{R}}_{00}
+e^a{\bf \widehat{\bf R}}_{ia}^{{}{(A)}}
-\Gamma^a{}_i{}^{\beta}{\bf \widehat{\bf R}}_{\beta a}-\Gamma_a{}^{a\beta}{\bf \widehat{\bf R}}_{i\beta}
{}{-L_i(C,\widehat{\bf R})}
\end{align*}
which proves \eqref{e0R0b}.

Employing the identity \eqref{contrsecondBianchi} once more, for $\mu=0$, we have
\begin{align*}
 e_0{\bf \widehat{R}}_{00}\overset{\eqref{Ddef}}{=}&\,{\bf \widetilde{D}}_0{\bf \widehat{R}}_{00}=-\frac{1}{2}e_0{\bf \widehat{R}}+{\bf \widetilde{D}}^a{\bf \widehat{R}}_{0a}{}{-L_0(C,\widehat{\bf R})}\\
\overset{\eqref{Rhat}}{=}&\,2e_0{\bf \widehat{R}}_{00}+e^a{\bf \widehat{R}}_{0a}-K^{ab}{\bf \widehat{R}}_{ba}^{{}{(S)}}-\text{tr}K{\bf \widehat{R}}_{00}-\Gamma_a{}^{ab}{\bf \widehat{R}}_{0b}{}{-L_0(C,\widehat{\bf R})}
\end{align*}
Solving for $e_0{\bf \widetilde{R}}_{00}$ and using \eqref{Rijeq}, \eqref{Rb0R0bhat}, we obtain \eqref{R00eq}.

{}{Going back to \eqref{hataRic}, we put $\alpha=i,\beta=j$ and use \eqref{Cijb} to keep only the spatial part of the identity.} Differentiating both sides in $e_0$ and using the commutation formula \eqref{comm} we compute:
{}{
\begin{align}\label{RAhatder}
-2e_0{\bf\widehat{R}}^{{}{(A)}}_{ij}=&\,e_0\bigg[\widetilde{D}^b C_{ijb}+\widetilde{D}_i C_{jb}{}^b+ \widetilde{D}_j C_{b i}{}^b
+C_{ij}{}^l C_{lb}{}^b+C_{b i}{}^l C_{lj}{}^b+C_{jb}{}^l C_{li}{}^b-C_{jb}{}^l\Gamma_{l i}{}^b-C_{b i}{}^l\Gamma_{l j}{}^b\bigg]\\
\notag=&\,e^be_0C_{ijb}+e_ie_0C_{jb}{}^b+e_je_0C_{bi}{}^b-K^{bl}e_lC_{ijb}-K_i{}^le_lC_{jb}{}^b-K_j{}^le_lC_{bi}{}^b+e_0L_{ij}(C)
\end{align}
where $L_{ij}(C)$ is given by \eqref{Lij}}.
We rewrite the second line in \eqref{RAhatder} by plugging in \eqref{e0Cijb}:
\begin{align}\label{RAhatder2}
\notag&e^be_0C_{ijb}+e_ie_0C_{jb}{}^b+e_je_0C_{bi}{}^b-K^{bl}e_lC_{ijb}-K_i{}^le_lC_{jb}{}^b-K_j{}^le_lC_{bi}{}^b\\
\notag=&\,e^b\bigg[{}{K_b{}^lC_{ijl}}-K_i{}^lC_{ljb}-K_j{}^lC_{ilb}-\delta_{ib}{\bf \widehat{R}}_{j0}+\delta_{jb}{\bf \widehat{R}}_{i0}\bigg]
+e_i\bigg[{}{K^{bl}C_{jbl}}-K_j{}^lC_{lb}{}^b-K^{bl}C_{jlb}-{\bf \widehat{R}}_{j0}+3{\bf \widehat{R}}_{j0}\bigg]\\
\notag&+e_j\bigg[{}{K^{bl}C_{bil}}-K^{bl}C_{lib}-K_i{}^lC_{bl}{}^b-3{\bf \widehat{R}}_{i0}+{\bf \widehat{R}}_{i0}\bigg]-K^{bl}e_lC_{ijb}-K_i{}^le_lC_{jb}{}^b-K_j{}^le_lC_{bi}{}^b\\
=&\,e_i{\bf\widehat{R}}_{j0}-e_j{\bf\widehat{R}}_{i0}
{}{+K^{bl}(e_bC_{ijl}+e_iC_{jbl}+e_jC_{bil})}
-K_i{}^l(e^bC_{ljb}+e_jC_{bl}{}^b+e_lC_{jb}{}^b)\\
\notag&-K_j{}^l(e^bC_{ilb}+e_iC_{lb}{}^b+e_lC_{bi}{}^b)
-K^{bl}(e_iC_{jlb}+e_jC_{lib}+e_lC_{ijb})
{}{+C_{ijl}e^bK_b{}^l+C_{jbl}e_iK^{bl}+C_{bil}e_jK^{bl}}\\
\notag&-C_{ljb}e^bK_i{}^l-C_{ilb}e^bK_j{}^l
-C_{lb}{}^be_iK_j{}^l-C_{jlb}e_iK^{bl}
-C_{lib}e_jK^{bl}-C_{bl}{}^be_jK_i{}^l
\end{align}
On the other hand, from \eqref{hataRic} and the first Bianchi identity \eqref{hatfBian}, the spatial derivatives of $C_{ijb}$ in \eqref{RAhatder2} can be replaced by
\begin{align}\label{RAhatder3}
\notag&{}{K^{bl}(e_bC_{ijl}+e_iC_{jbl}+e_jC_{bil})}-K_i{}^l(e^bC_{ljb}+e_jC_{bl}{}^b+e_lC_{jb}{}^b)
-K_j{}^l(e^bC_{ilb}+e_iC_{lb}{}^b+e_lC_{bi}{}^b)\\
&-K^{bl}(e_iC_{jlb}+e_jC_{lib}+e_lC_{ijb})\\
\notag=&\,{}{2K_i{}^l{\bf\widehat{R}}_{lj}^{{}{(A)}}}+2K_j{}^l{\bf\widehat{R}}_{il}^{{}{(A)}}
+K^{bl}{\bf\widehat{R}}_{(ijl)b}{}{-K^{bl}\widehat{\bf R}_{(ijb)l}}{}{+K_i{}^lL_{lj}(C)+K_j{}^lL_{il}(C)+K^{bl}[L_{ijlb}(C)-L_{ijbl}(C)]},
\end{align}
where $L_{ijbl}$ is given by \eqref{Lijbl}.

Summarizing \eqref{RAhatder}-\eqref{Lijbl} gives \eqref{RijAeq} and completes the proof of the lemma.
\end{proof}
{}{In the presence of a timelike, totally geodesic, boundary, the boundary conditions \eqref{bdcond} yield boundary conditions for certain components of the modified Ricci curvature $\widehat{R}_{\alpha\beta}$. In particular, we have:}
\begin{lemma}\label{lem:bdcond.Ric}
{}{The spacetime metric ${\bf g}$ induced by the solution to the boundary problem for \eqref{e0fij}, \eqref{e0Kijnew}, \eqref{e0Gammaijbnew}, subject to \eqref{bdcond}, as described above, satisfies: }
\begin{align}\label{bdcondRic}
{\bf \widehat{R}}_{03}={\bf \widehat{R}}_{30}=0,\qquad {\bf\widehat{R}}_{B3}^{{}{(A)}}={\bf\widehat{R}}_{3B}^{{}{(A)}}=0,\qquad\text{on $\mathcal{T}$}.
\end{align}

\end{lemma}
\begin{proof}
All subsequent computations are restricted to the boundary $\mathcal{T}$.
The first boundary condition follows 
by setting $b=3$ in \eqref{Rb0R0bhat} and using the boundary condition \eqref{bdcond}:
\begin{align*}
{\bf \widehat{R}}_{03}=-{\bf \widehat{R}}_{30}
=-e_0\Gamma_{i3}{}^i-K_i{}^c\Gamma_{c3}{}^i=-e_0\Gamma_{B3}{}^B-K_{A3}\Gamma_{33}{}^A-K_A{}^B\Gamma_{B3}{}^A=0.
\end{align*}
For the second boundary condition, we first notice that by \eqref{Rijeq} it holds 
\begin{align*}
{\bf\widehat{R}}_{B3}=-{\bf\widehat{R}}_{3B}\qquad\Rightarrow\qquad{\bf\widehat{R}}_{B3}^{{}{(A)}}=-{\bf\widehat{R}}_{3B}^{{}{(A)}}={\bf\widehat{R}}_{B3}.
\end{align*}
Contracting \eqref{Gausshat} in $(a;b)$ and setting $i=B,j=3$, we obtain
\begin{align*}
{\bf\widehat{R}}_{B3}-{\bf\widehat{R}}_{0B03}=&\,\widehat{R}_{B3}+\text{tr}KK_{B3}-K_3{}^aK_{Ba}\\
\tag{by \eqref{Rijhat}, $i=B,j=3$}=&\,e^b\Gamma_{B3b}-e_B\Gamma^b{}_{3b}
{}{+\Gamma^b{}_B{}^c\Gamma_{c3b}+\Gamma^b{}_b{}^c\Gamma_{B3c}}\\
&+\text{tr}KK_{B3}-K_3{}^aK_{Ba}\\
\tag{by \eqref{R0b0i=R0i0b}}{\bf\widehat{R}}_{B3}=&\,
e^C\Gamma_{B3C}-e_B\Gamma^C{}_{3C}
+\Gamma^C{}_B{}^D\Gamma_{D3C}
+\Gamma^C{}_{B3}\Gamma_{33C}+\Gamma^a{}_a{}^C\Gamma_{B3C}\\
&+\text{tr}KK_{B3}-K_3{}^aK_{Ba}+e_0K_{B3}+K_B{}^AK_{3A}+K_{B3}K_{33}
\end{align*}
Every term in the preceding RHS vanished by virtue of the boundary condition \eqref{bdcond}, which implies the vanishing of $\widehat{\bf R}_{B3}$ and hence that of $\widehat{\bf R}^{{}{(A)}}_{B3}$.
\end{proof}

\subsection{Final step}\label{subsec:finstep}

The equations \eqref{e0Cijb}-\eqref{Rijeq} constitute a {\it linear} first order {\it symmetric} hyperbolic system (see also Remark \ref{rem:Ricsym}) for the variables ${\bf \widehat{R}}_{\mu0},{\bf\widehat{R}}_{ij}^{{}{(A)}},C_{ijb}$, which in the presence of a timelike boundary also satisfy the conditions \eqref{bdcondRic}. As an immediate implication, we conclude that the solution $K_{ij},\Gamma_{ijb},f_i{}^j$, $f^b{}_j$ to the reduced equations \eqref{e0Kijnew}, \eqref{e0Gammaijbnew}, \eqref{e0fij}, \eqref{e0fijinv}, is indeed a solution to the EVE. More precisely, we have:
\begin{proposition}\label{prop:EVEsol}
Consider a solution to the reduced equations \eqref{e0Kijnew}, \eqref{e0Gammaijbnew}, \eqref{e0fij}, \eqref{e0fijinv}, such that
\begin{enumerate}
\item $K_{ij},\Gamma_{ijb},f_i{}^j,f^b{}_j\in L^\infty_tH^s$, $s\ge3$, for the classical Cauchy problem;
\item $K_{ij},\Gamma_{ijb},f_i{}^j,f^b{}_j\in L^\infty_tB^s$, $s\ge7$, subject to \eqref{bdcond}, for the boundary value problem.
\end{enumerate}
Then the variables ${\bf \widehat{R}}_{\mu\nu},C_{ijb}$ vanish. In particular, ${\bf \widetilde{D}}$ is the Levi-Civita connection ${\bf D}$ of ${\bf g}$. Moreover, ${\bf g}$ satisfies the EVE and in the case 2. the boundary is totally geodesic.
\end{proposition}
\begin{proof}
The coefficients in \eqref{e0Cijb}-\eqref{Rijeq} depend on $K_{ij},\Gamma_{ijb},f_i{}^j,f^b{}_j$ and their first spatial derivatives. Hence, they are bounded, provided up to three of their spatial derivatives are bounded in $L^2$. This is consistent with the spaces $L^\infty_tH^s$, for $s\ge3$, and $L^\infty_tB^s$, for $s\ge7$.

In the absence of a boundary, the symmetry of the system \eqref{e0Cijb}-\eqref{Rijeq} implies uniqueness of solutions (via a standard energy estimate). Since $C_{ijb}$ vanishes on the initial hypersurface, we have that ${\bf\widetilde{D}}={\bf D}$. {}{By virtue of \eqref{hataRic} (for $\alpha=i,\beta=j$) and \eqref{Cijb}, we have that $\widehat{\bf R}_{ij}^{{}{(A)}}\big|_{\Sigma_0}=0$.} Also, the validity of the constraints, together with the formula \eqref{Rijeq}, implies {}{${\bf\widehat{R}}_{\mu0}\big|_{\Sigma_0}=0$}, see \eqref{hatRb0}, \eqref{contrGausshat}. Hence, ${\bf\widehat{R}}_{\mu0},{\bf\widehat{R}}_{ij}^{{}{(A)}},C_{ijb}$ vanish everywhere and ${\bf\widetilde{D}}={\bf D}$. By \eqref{Rijeq}, ${\bf\widehat{R}}_{ij}^{{}{(S)}}=0$, and hence, ${\bf\widehat{R}}_{\mu\nu}={\bf R}_{\mu\nu}=0$.

In the presence of a timelike boundary, we notice that in a typical $L^2$-energy estimate for \eqref{e0Cijb}-\eqref{RijAeq}, the arising $\mathcal{T}$-boundary terms equal
\begin{align}\label{bdtermRic}
\int_{{}{\mathcal{S}_t}}{\bf \widehat{R}}_{00}{\bf \widehat{R}}_{30}+{\bf\widehat{R}}_{i3}^{{}{(A)}}{\bf \widehat{R}}^i{}_0\mathrm{vol}_{{}{\mathcal{S}_t}}=\int_{{}{\mathcal{S}_t}}{\bf \widehat{R}}_{00}{\bf \widehat{R}}_{30}+{\bf\widehat{R}}_{B3}^{{}{(A)}}{\bf \widehat{R}}^B{}_0+{\bf\widehat{R}}_{33}^{{}{(A)}}{\bf \widehat{R}}_{30}\mathrm{vol}_{{}{\mathcal{S}_t}}\overset{\eqref{bdcondRic}}{=}0.
\end{align}
Therefore, an energy estimate closes and the previous argument applies. Since $\widetilde{\bf D}={\bf D}$ is the actual Levi-Civita connection of ${\bf g}$, the variables $K_{ij},\Gamma_{ijb}$ are the true connection coefficients of the orthonormal frame $\{e_\mu\}_0^3$, given by \eqref{Kij}, \eqref{Gammaijb}. Hence, the geometric formulas \eqref{bdid} are valid, where $\chi_{0A},\chi_{AB}$ are the components of the actual second fundamental form $\chi$ of $\mathcal{T}$, which vanish by virtue of the condition \eqref{bdcond}. The component $\chi_{00}={\bf g}({\bf D}_{e_0}e_3,e_0)=-{\bf g}(e_3,{\bf D}_{e_0}e_0)$ vanishes, since $e_0$ is geodesic. We conclude that $\chi\equiv0$, i.e., $\mathcal{T}$ is totally geodesic.
\end{proof}

\begin{proof}[Proof of Theorems \ref{thmA}, \ref{thmB}]
It is a combination of Propositions \ref{prop:locex}, \ref{prop:bdlocex}, \ref{prop:EVEsol}. We note in particular that geometric uniqueness is immediate from the homogeneity of our boundary conditions. After setting up the geodesic gauge in any vacuum spacetime with totally geodesic timelike boundary, the relevant connection coefficients will vanish, in which case the uniqueness statement for the reduced system of equations applies to solutions with the same initial data.
\end{proof}


\begin{thebibliography}{99}

\bibitem{AA} Z. An and M. T. Anderson,
\textit{On the initial boundary value problem for the vacuum Einstein equations and geometric uniqueness},
arXiv:2005.01623

\bibitem{AndYork} A. Anderson and J. W. York, Jr, 
\textit{Fixing Einstein's equations}, Phys. Rev. Lett. {\bf 82} (1999), no. 22, 4384-4387. 

\bibitem{AndMon} L. Andersson and V. Moncrief, 
\textit{Elliptic-hyperbolic systems and the Einstein equations},
Ann. Henri Poincar\'e {\bf 4} (2003), no. 1, 1-34.

\bibitem{CarVal} D. A. Carranza and J. A. Valiente Kroon,
       \textit{Construction of anti--de Sitter--like spacetimes using the
              metric conformal Einstein field equations: the vacuum case},
Classical Quantum Gravity {\bf 35} (2018), no. 24, 245006, 34. 

\bibitem{ChoqRug} Y. Choquet-Bruhat and T. Ruggeri,
\textit{Hyperbolicity of the 3+1 system of Einstein equations},
Comm. Math. Phys. {\bf 89} (1983), no. 2, 269-275. 

\bibitem{EncKam} A. Enciso and N. Kamran,
    \textit{Lorentzian Einstein metrics with prescribed conformal
              infinity}, J. Differential Geom. {\bf 112} (2019), no. 3, 505-554.
              
\bibitem{FLuk} G. Fournodavlos and J. Luk, 
\textit{Asymptotically Kasner-like singularities}, arXiv:2003.13591.

\bibitem{FSm} G. Fournodavlos and J. Smulevici, 
\textit{On the initial boundary value problem for the Einstein vacuum equations in the maximal gauge}, arXiv:1912.07338. 

\bibitem{Fried95} H. Friedrich,
   \textit{Einstein equations and conformal structure: existence of Anti-de Sitter-type space-times}, J. Geom. Phys. {\bf 17} (1995), no. 2, 125-184.

\bibitem{Fried96} H. Friedrich,
\textit{Hyperbolic reductions for Einstein's equations},
Class. Quantum Grav. {\bf 13} (1996), no. 6, 1451-1469.

\bibitem{Fried09} H. Friedrich,
\textit{Initial boundary value problems for Einstein's field
equations and geometric uniqueness}, Gen. Relativ. Gravit. {\bf 41} (2009), 1947-1966.

\bibitem{FriedNag} H. Friedrich and G. Nagy,
\textit{The initial boundary value problem for Einstein's vacuum field equation}, 
Comm. Math. Phys. {\bf 201} (1999), no. 3, 619-655.

 \bibitem{FriReu} S. Frittelli and O. Reula,
     \textit{On the {N}ewtonian limit of general relativity},
   Comm. Math. Phys. {\bf 166}, no. 2, 221-235. 
   
 \bibitem{HFS} L. A. Hau, J. L. Flores and M. S\'anchez,
\textit{Structure of globally hyperbolic spacetimes with timelike boundary},
arXiv:1808.04412. 


\bibitem{WK}
W.P.A. Klingenberg,
\textit{Riemannian geometry}, De Gruyter Studies in Mathematics Vol 1 (1995), Second Edition.


\bibitem{KRSW}
H. O. Kreiss, O. Reula, O. Sarbach and J. Winicour, 
\textit{Boundary conditions for coupled quasilinear wave equations with application to isolated systems},
Comm. Math. Phys. {\bf 289} (2009), no. 3, 1099-1129.

\bibitem{RodSp} I. Rodnianski and J. Speck, 
\textit{Stable big bang formation in near-FLRW solutions to the Einstein-scalar field and Einstein-stiff fluid systems},
Selecta Math. (N.S.) {\bf 24} (2018), no. 5, 4293-4459.

\bibitem{SarTig} O. Sarbach and M. Tiglio, 
\textit{Boundary conditions for Einstein's field equations: mathematical and numerical analysis}, J. Hyperbolic Differ. Equ. {\bf 2} (2005), no. 4, 839-883


\bibitem{SarTig2} O. Sarbach and M. Tiglio,
\textit{Continuum and Discrete Initial-Boundary Value Problems and Einstein's Field Equations}, M. Living Rev. Relativ. (2012) 15: 9. https://doi.org/10.12942/lrr-2012-9.
              
 






\end{thebibliography}
\end{document}